\documentclass[10pt]{amsart}
\usepackage[margin=1in]{geometry}

\usepackage{tikz}
\usepackage{tikz-cd}
\usepackage{multirow,booktabs}
\usepackage[citation-order]{amsrefs}
\usepackage{hyperref}
\usepackage{amsthm}
\usepackage{mathtools}
\usepackage{fancyhdr}
\usepackage{subcaption}
\usepackage{amssymb}
\usetikzlibrary{3d,calc}

\DeclareMathOperator{\sgn}{sgn}

\DeclareMathOperator{\MCG}{MCG}
\DeclareMathOperator{\PMCG}{PMCG}
\DeclareMathOperator{\PSL}{PSL}
\DeclareMathOperator{\Dev}{Dev}
\DeclareMathOperator{\arccosh}{arccosh}
\DeclareMathOperator{\GL}{GL}
\DeclareMathOperator{\Aut}{Aut}
\DeclareMathOperator{\Hom}{Hom}

\begin{document}

\theoremstyle{definition}
\newtheorem{exmp}{Example}[section]
\newtheorem{defn}{Definition}[section]
\newtheorem{con}{Conjecture}[section]
\newtheorem{rem}{Remark}[section]
\theoremstyle{plain}
\newtheorem{thm}{Theorem}[section]
\newtheorem{lem}{Lemma}[section]

\newtheorem{cor}{Corollary}[section]

\title{Mutation Invariant Functions on Cluster Ensembles}

\author{Dani Kaufman}
\address{University of Copenhagen \\
         Department of Mathematical Sciences \\
         2100 Copenhagen \o, Denmark }
\email{dk@math.ku.dk}
\thanks{Special thanks to Zachary Greenberg and Christian Zickert for  helpful advice and encouragement and to Hugh Thomas for useful comments and suggestions.}


%\keywords{Cluster Algebras, Cluster Ensembles, Markov Numbers, Teichmuller Spaces, Somos Sequences.}

\date{}

%\dedicatory{}

\begin{abstract}

We define the notion of a mutation invariant function on a cluster ensemble with respect to a group action of the cluster modular group on its associated function fields. We realize many examples of previously studied functions as elements of this type of invariant ring and give many new examples. We show that these invariants have geometric and number theoretic interpretations, and classify them for ensembles associated to affine Dynkin diagrams. The primary tool in used in this classification is the relationship between cluster algebras and the Teichmüller theory of surfaces. 

\smallskip
\smallskip
\noindent \textbf{Cluster Algebras, Cluster Ensembles, Markov Numbers, Teichmuller Spaces, Somos Sequences.}

\end{abstract}

\maketitle

\section{Introduction}

Cluster algebras have emerged as important structures that underlie many different algebraic objects, including the algebraic structure of Teichm\"uller space and the structure of scattering amplitudes in $\mathcal{N}=4$ supersymmetric Yang Mills theory. In the cases where the cluster algebra at hand is of finite type, all of the important combinatorial and algebraic data that can be derived from it can be easily computed and studied. Complications begin to arise when the algebra is no longer of finite type, as can be seen when attempting to compute scattering amplitudes of more than 7 particles in $\mathcal{N}=4$ SYM theory. This paper is an attempt to introduce and study a collection of functions assigned to any cluster algebra, or more generally any ``cluster ensemble'', that are invariant under certain mutation sequences.

Several examples of cluster ensemble invariants have already been noted in the literature, but they have not been studied in a unified way. Of these examples, the invariant of the Markov quiver, shown in example \ref{exmp:markov}, is probably the most well known and studied. This function is directly related to the classical Markov Diophantine equation $x^2+y^2+z^2 = 3xyz$.  
Some examples of similar Diophantine equations arising from cluster were studied in \cite{Lampe}, and we recall such an equation in example \ref{exmp:bc24}. The invariants of the Somos 4 and 5 sequences studied in \cite{Somos} are further examples, shown in example \ref{exmp:somos}. The invariants in each of these examples provide useful tools for understanding the number theoretical properties of each sequence. 

All of these previously studied functions are invariant under the action of a single mutation. The primary new ingredient of this paper is the introduction of an action of the automorphism group of the exchange complex of a cluster ensemble, or ``cluster modular group'', on the function fields of the ensemble. We show that each of the previously known examples can be realized as elements of invariant rings of this group action. We use this definition to generalize previously studied examples to invariants associated to more complicated subgroups of the cluster modular group.

% We hope that this theory provides a method to understand and tame the infiniteness of certain cluster ensembles. 
% When the cluster modular group is infinite, we naturally have some repeating structures underlying the mutation structure of the associated ensemble. If we can find invariants for a finite index normal subgroup of the cluster modular group, then we may consider these functions as naturally being defined on a finite quotient of the exchange complex. 

The primary structural theorem proved in this paper is a classification of invariants associated to an ``affine'' cluster ensemble i.e. a cluster ensemble associated to an affine $ADE$ Dynkin diagram. Affine cluster ensembles are essentially the simplest ensembles with an infinite cluster modular group, and the properties of their invariants should be indicative of the structure seen in general invariants. The proof of this classification relies heavily on a geometric interpretation of affine ensembles and associated cluster modular groups. Essentially, the classification follows by relating invariants in the $\tilde{A}_1$ case to the ring of functions on the $\PSL(2,\mathbb{Z})$ character variety associated with an annulus. 
The invariants classified in Section \ref{sec:affine} are invariant for the same finite index normal subgroup defined in \cite{kauf:Modular_groups} of affine type cluster modular groups which is used to define ``affine generalized associahedra''. Thus, the classification of these functions provides a basis of the field of rational functions on each affine generalized associahedra. 

We may also use the invariance of a function to study the limiting behavior of the cluster variables found along a mutation path.
Recently, a cluster algebraic interpretation of some of the symbols of 8 particle $\mathcal{N}=4$ SYM scattering amplitudes has been studied in \cite{SYMbols} and \cite{arkanihamed2019nonperturbative}. These symbols are related to the limiting structure of an affine cluster algebra and were studied using an invariant of the $\widetilde{A}_1$ affine cluster algebra. We briefly show a similar analysis of this limiting behavior in example \ref{exmp:a11}.

We give strong evidence through an abundance of examples of a deeper theory underlying the existence and structure of cluster ensemble invariants. We show in each example a correspondence between invariants on the $\mathcal{A}$ and $\mathcal{X}$ spaces via denominator vectors. We also show evidence that there should be a basis of $\mathcal{A}$ invariants such that the cluster modular group acts on this basis by positive Laurent polynomials. This is a generalization of the Laurent phenomenon of a cluster algebra. These conjectures together are a sort of generalization of the duality conjectured in \cite{FockGonch:1}.  
We strongly suspect that the existence of invariants will always be related to some other important manifestation of the ensemble, be it geometric or number theoretic. 

\subsection{Structure of this paper}
Section \ref{sec:background} recalls the background necessary to give a concise definition of a cluster ensemble invariant. 
Section \ref{sec:examples} contains many examples of invariants for various cluster algebras. 
In Section \ref{sec:structures}, we prove some basic structural theorems about invariants for general cluster algebras. 
Section \ref{sec:geometry} introduces some ideas of hyperbolic geometry and Teichm\"uller theory into our analysis of invariants, which explains the origin of and relationships between many of the examples in section 3. 
%It is well known that the Markov cluster algebra has an interpretation in terms of the hyperbolic geometry of the once punctured torus \cite{Caroline:markov}. We briefly recall this idea in example  \ref{exmp:hypmarkov} and we apply generalizations of this idea to study new invariants.  
In section \ref{sec:affine}, we classify the invariants of algebras of affine cluster ensembles. 
And finally, section \ref{sec:conjectures} covers some conjectures about more complicated structures underlying invariants in general.

\section{Background}\label{sec:background}

We will recall the basic notions of a cluster ensemble from \cite{FockGonch:2}.  This consists of a pair of positive spaces ($\mathcal A$,$\mathcal X$) along with a map $\rho$ from $\mathcal A$ to $\mathcal X$, a notion of seeds and mutations, and a group, $\Gamma $, called the ``cluster modular group'' that acts by automorphisms of the entire structure. We will, however, greatly simplify these definitions to emphasize a more concrete and computational framework in which to introduce and study invariants. We define an action of $\Gamma$ on the ring of rational functions on each of the $\mathcal A$ and $\mathcal X$ spaces. Our notion of an invariant function will be in regards to this group action.

\subsection{Quivers}

Our first simplification is to only consider cluster ensembles of ``geometric'' type, meaning that we can use quivers to define our ensembles.

\begin{defn} 

A \emph{quiver} is a directed and weighted graph with no self loops or 2 cycles. We think of a quiver as a graphical representation of a matrix, $M$, called the \emph{exchange matrix} which has entries $[\epsilon_{ij}]$ equal to the number of arrows from node $i$ to node $j$. We denote the set of nodes of $Q$ by $N(Q)$ and we usually refer to them by their index in this set. 
\end{defn}

We will allow non-symmetrically weighted arrows between nodes, to account for non skew-symmetric exchange matrices. However, in this case we require that the exchange matrix is skew-symmetrizable, meaning that there is a diagonal matrix, $D$, such that $MD^{-1}$ is skew-symmetric. The matrix $D$ associates to each node a multiplier, $d_i$.

In the various diagrams of quivers in the paper, we label the multipliers of our nodes as superscripts and label the edges by the weights, where no label means weight 1, a double arrow means weight 2, a single label means symmetric weights and a pair of weights for an arrow from node $i$ to node $j$ is $ -\epsilon_{ji} , \epsilon_{ij}$\footnote{This ordering is meant to agree with Berhnard Keller's java applet \href{https://webusers.imj-prg.fr/~bernhard.keller/quivermutation/}{Quiver Mutation in JavaScript}.}.

Quivers will underlie the coordinate atlases of our cluster ensembles and mutations will give the transition maps between coordinate charts.

\begin{defn}
A \emph{mutation} of a quiver $Q$ at node $i$, written $\mu_i(Q)$, generates a new quiver by the following 2 operations
\begin{enumerate}
	\item For every pair of nodes $j, k$ with weighted arrows $j \xrightarrow{a ,b} i \xrightarrow{c,d} k  $,  add an arrow of weight $ac,bd$ from $j$ to $k$.
	\item Swap the direction and weights of all the arrows coming into and out of node $i$.
\end{enumerate}
\end{defn}

\begin{defn}
Two quivers are \emph{isomorphic} if there is a weighted graph isomorphism between them.
The set of all quivers up to isomorphism obtained by all possible sequences of mutations of a particular quiver, $Q$, is called the \emph{mutation class} of $Q$. A particular representative of the mutation class is sometimes referred to as the \emph{type} of the quiver. 
\end{defn}

\begin{defn}
It is useful to occasionally include nodes in a quiver that we do not allow mutations at. These nodes are called \emph{frozen} and we write $Q^\mu$ for the subquiver of $Q$ consisting of non frozen nodes.
\end{defn}

\begin{exmp}
\label{exmp:g21}
Let $Q_{\widetilde{G}_2}$ be the quiver shown in figure \ref{fig:qg2}. We read the pair of arrow weights as ``1 in , 3 out"  meaning that the matrix associated to this quiver is 
$\begin{bmatrix}
 0 & 3 & 0\\
 -1 & 0 & 1 \\
 0 & -1 & 0 
\end{bmatrix}$. We may take $D = \text{diag}(1,3,3)$ so that nodes 2 and 3 get multipliers of 3. This quiver is associated to a root system of affine type $\widetilde{G}_2$, where node 1 is associated with the larger root and nodes 2 and 3 are associated with the smaller roots. $\mu_2(Q_{\widetilde{G}_2})$ is shown in figure \ref{fig:g233mu}. The exchange matrix changes to 
$\begin{bmatrix}
 0 & -3 & 3\\
 1 & 0 & -1 \\
 -1 & 1 & 0 
\end{bmatrix}$ after mutating.

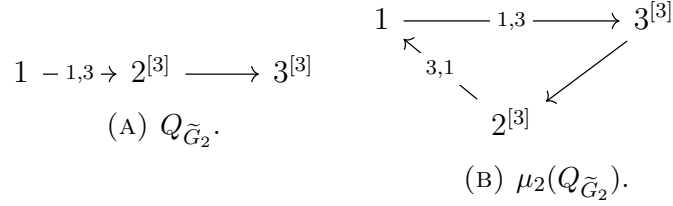
\begin{figure}[t]
    \centering
    \begin{subfigure}[]{0.3\textwidth}
        \centering
    \begin{tikzcd}
        1 \arrow[r, "{1,3}" description] & 2^{[3]} \arrow[r] & 3^{[3]}
        \end{tikzcd}
    \caption{$Q_{\widetilde{G}_2}$.}
    \label{fig:qg2}
    \end{subfigure} 
    \begin{subfigure}[]{0.3\textwidth}
    \begin{tikzcd}
         1 \arrow[rr, "{1,3}" description] &                                   & 3^{[3]} \arrow[ld] \\
                                  & 2^{[3]} \arrow[lu, "{3,1}" description] &             
    \end{tikzcd}
    \caption{$\mu_2(Q_{\widetilde{G}_2})$.}
    \label{fig:g233mu}
    \end{subfigure}
    \caption{Two quivers of type $\widetilde{G}_2$}
\end{figure}

\end{exmp}

\subsection{Positive spaces}
We will simplify the notion of positive space of \cite{FockGonch:2}.

\begin{defn}
A positive variety will simply mean a real manifold homeomorphic to $(\mathbb{R}^{>0})^k$ along with a distinguished coordinate chart. A map between positive varieties is called a \emph{positive map} and satisfies the property that pullback of the distinguished coordinate variables is a subtraction free rational function.
\end{defn}

\begin{defn}

A positive space is a collection of positive varieties along with invertible positive maps between them. 

\end{defn}

Equipped with the notions of quivers and mutations, we can define the pair $(\mathcal A_Q,\mathcal X_Q)$ of positive spaces associated to a quiver.

\subsection{The $\mathcal{A}$ space}

Let $Q$ be a quiver with mutable and frozen nodes 
\begin{equation}
    N_1, \cdots, N_n , F_{n+1}, \cdots, F_{n+f}. 
\end{equation}
Let 
\begin{equation}
    \{a_1,\cdots,a_n,a_{n+1},\cdots, a_{n+f}\}
\end{equation}
be algebraically independent generators of the function field
$\mathbb{R}(a_1,\cdots, a_{n+f})$. We associate these variables variables to the nodes of $Q$ in an obvious way. 

\begin{defn}
The pair of a quiver along with variables associated to its nodes is called an \emph{$\mathcal{A}$ seed}, and we call these variables \emph{$\mathcal{A}$ coordinates}. The $\mathcal{A}$ coordinates associated to the frozen nodes are called \emph{coefficients}. 
\end{defn}

Often, we will simply refer to a seed by the quiver underlying it, assuming that there are independent variables associated to its nodes. We write $\mathcal{A}_{|Q|}$ for the positive variety with coordinate variables given by the seed $Q$.

Mutation of $Q$ a node $i$ produces a new seed consisting of a new quiver $\mu_i(Q)$ and a new collection of variables by exchanging $a_i$ with $a_i'$ satisfying 
\begin{equation}
    a_i' = \dfrac{1}{a_i} \left(\prod_{j \xrightarrow{\alpha,\beta} i}{a_j}^{\alpha} + \prod_{i \xrightarrow{\gamma,\delta} k}{a_k}^\delta\right).
\end{equation}
Each collection of $\mathcal A$ coordinates generated in this way is called a cluster, and the $\mathbb{R}$-subalgebra of 
\begin{equation}
    \mathbb{R}(a_1,\cdots,a_n)[a_{n+1},\cdots, a_{n+f}]
\end{equation}
generated by the clusters of $\mathcal A$ coordinates obtained from all possible sequences of mutations at non-frozen nodes is the cluster algebra associated to $Q$.

Mutation of seeds provide isomorphisms of positives varieties given by 
\begin{align}
    \mu_i: &\mathcal{A}_{|Q|} \rightarrow \mathcal{A}_{|\mu_i(Q)|} \\
    \mu_i^*(a_j') = &
    \begin{cases}
     a_i^{-1} \left(\prod\limits_{j \xrightarrow{\alpha,\beta} i}{a_j}^{\alpha} + \prod\limits_{i \xrightarrow{\gamma,\delta} k}{a_k}^\delta\right) & \text{if } i=j \\
     a_j & \text{otherwise}
    \end{cases}
\end{align}

\begin{defn}

$\mathcal A_Q$ is the positive space consisting of the positive varieties associated to every seed generated by mutation of $Q$ and maps given by the mutation isomorphism for each mutation.
\end{defn}

\subsection{The $\mathcal{X}$ space}

The $\mathcal X$ space will have a definition similar to that of the $\mathcal A$ space, but with a different exchange rule. In the case that $Q$ has frozen nodes, we will not include extra variables, and we may essentially ignore these nodes in the definition. 

Given a quiver, $Q$, with mutable nodes $N_1, \cdots, N_n$, we associate independent variables $ x_1, \cdots , x_n$ called $\mathcal{X}$ coordinates. This pair is called an $\mathcal{X}$ seed. Mutation of $Q$ at node $i$ produces a new $\mathcal{X}$ seed consisting of a quiver $\mu_i(Q)$ and a new collection of $\mathcal{X}$ coordinates by the following rule:
\begin{equation}
    x'_j = 
    \begin{cases}
    x_i^{-1} & \text{if }  i=j \\
    x_j(1+x_i^{-\sgn{\epsilon_{ji}}})^{-\epsilon_{ji}} & \text{if } \epsilon_{ji}\neq 0 \\
    x_j & \text{otherwise}
    \end{cases}
\end{equation}

Again, mutation provides an isomorphism of positive varieties $\mu_i:\mathcal{X}_{|Q|} \rightarrow \mathcal{X}_{|\mu_i(Q)|} $ with $\mu_i^*(x_j')$ given by the above mutation rule.

\begin{defn}
$\mathcal X_Q$ is the positive space consisting of the positive varieties associated to every seed generated by mutation and maps given by the above mutation isomorphism for each mutation.
\end{defn}

\begin{defn}
Given an initial quiver, $Q$, we write ($\mathcal{A}_Q$,$\mathcal{X}_Q$) for the pair of positive spaces generated by using $Q$ as a seed for each space, where $\mathcal{X}_Q=\mathcal{X}_{Q^\mu}$ by abuse of notation. This pair of spaces is the \emph{cluster ensemble} associated to $Q$. The \emph{rank} of this ensemble is $\#N(Q^\mu)$. We often drop the subscripts when $Q$ is implied.
\end{defn}

\begin{rem}
The positive varieties $\mathcal{A}_{|Q|}$ and $\mathcal{X}_{|Q|}$ associated to the initial seed $Q$ provide a base point of each space. We usually consider the coordinates of these particular varieties to be the coordinate functions for the entire space. 
\end{rem}

The final ingredient of a cluster ensemble is a map $\rho: \mathcal A_Q \rightarrow \mathcal X_Q $. $\rho$ has a simple definition in terms of the initial coordinates. If $(a_1,\cdots,a_n,a_{n+1},\cdots, a_{n+f})$ and $ (x_1,\dots,x_n)$ are the initial $\mathcal A$ and $\mathcal X$ coordinates and the exchange matrix of $Q$ is $[\epsilon_{ij}]$, then we have that $\rho^*(x_i) = \displaystyle{\prod_i a_i^{\epsilon_{ij}}}$. $\rho$ commutes with mutations and provides a map of positive spaces. 

\subsection{The type of a cluster ensemble}

The combinatorial properties of a cluster ensemble are controlled by the mutation class of the quiver, $Q^\mu$, underlying its seeds. We refer a cluster ensemble associated to any quiver in the mutation class of $Q$ as a ``type $Q$'' cluster ensemble.  

The most well known example of this is the fact that a cluster ensemble has finitely many clusters if and only if there is a quiver in its mutation class that is an orientation of a classical Dynkin diagram, see \cite{Fominzel:2}. We call these cluster ensembles ``finite type'' and refer to them by their Dynkin diagram. 

We will be interested in ``affine type'' cluster ensembles. These are associated to quivers with an orientation of an affine Dynkin diagram in their mutation classes e.g. the quiver of example \ref{exmp:g21} is an orientation of a $\widetilde{G}_2$ Dynkin diagram.  These are examples of quivers with a finite mutation class, and most of the examples given in this paper will be associated with this kind of quiver. 

Several examples will be associated with the doubly extended Dynkin diagrams of \cite{Saito:1}, which constitute almost all of the exceptional finite mutation class quivers of \cite{FST:finite_mutation_via_unfoldings}.

\subsection{The cluster modular group}

\begin{defn}
The automorphism group of the mutation structures of the $\mathcal A$ and $\mathcal X$ spaces defined above is called the \emph{cluster modular group}. 
\end{defn}

This can be considered as the group of symmetries of the ``exchange complex'' of the cluster ensemble, see section 2.4 of \cite{FockGonch:2}. 

We will give a conjecturally equivalent description of it here in terms of sequences of mutations and isomorphisms of quivers. This is the same as the notion of the ``special modular group'' of section 1 of \cite{FockGonch:2}. We will not worry about the possible distinction between the special cluster modular group and the actual cluster modular group.  

Fix a seed with quiver $Q$. Let $\tilde{\Gamma}_{Q} = \big{\{}\{P,\sigma\}\big{\}}$ where $P$ is a sequence of mutations and $\sigma$ is an isomorphism $Q^\mu \rightarrow P(Q)^\mu$. Because $P$ transforms $Q$ into an isomorphic quiver, we may compose these mutation paths. We write paths of mutations as a sequence of the node names read from left to right. The action of $\sigma$ as an element of the symmetric group is written with the standard left action.

Since each mutation is an involution, we have that $\tilde{\Gamma}_{Q}$ is a subgroup of $ (\mathbb{Z}/{2\mathbb{Z}}^{*k}) \rtimes S_k $ where $\mathbb{Z}/{2\mathbb{Z}}^{*k}$ is a free product of cyclic groups and the right hand side of the semidirect product acts by permuting the factors.

\begin{rem}
\label{rem:action}
$\tilde{\Gamma}_{Q}$ acts on the elements of the clusters of $\mathcal{A}$ and $\mathcal{X}$ coordinates. $P$ provides a path to a new cluster and $\sigma$ provides a map between the initial cluster elements and the final elements.

\end{rem}

\begin{defn}
The \emph{cluster modular group}, $\Gamma_Q$ is defined to be $\tilde{\Gamma}_{Q}$ modulo the subgroup generated by elements that act trivially on the $\mathcal{X}$ coordinates.
\end{defn}

%\begin{rem}
%In the definition above, we only consider the action on the $\mathcal{X}$ coordinates. This is due to the fact that frozen variables may change  In the case where there are coefficients on the $\mathcal{A}$ space, we may find that there are cluster modular group elements of the underlying quiver with no frozen nodes where $\sigma$ cannot be extended to an isomorphism that includes the frozen nodes. This will not be an important situation in this paper. See \cite{Fraser:quasi} for details about the types of issues that arise in this situation.  
%\end{rem}

\begin{exmp}
Let's examine the cluster modular group of the type $A_2$ cluster ensemble. The quiver of the $A_2$ cluster ensemble consists of two nodes, 1 and 2, with a single arrow between them. Mutation at node 1 produces a quiver that is isomorphic to the starting quiver (after permuting the nodes). We can represent this group element by $\gamma=\{1,(12)\}$. This element clearly generates the cluster modular group. The only relation is that $\gamma^5= e$. This relation comes from the fact that 5 consecutive applications of $\gamma$ reproduces the original cluster variables. Thus the cluster modular group is $\mathbb{Z}/(5\mathbb{Z})$.

\end{exmp}
For a more complete picture, we briefly review the ``cluster modular groupoid'' associated to a cluster ensemble.

\begin{defn}

We can associate to any cluster ensemble its \emph{cluster modular groupoid}, $\mathcal G_Q$, and view its cluster modular group as the group associated to $\mathcal G_Q$. The objects of $\mathcal G_Q$ are given by the isomorphism classes of quivers in the mutation class of $Q$, and the maps are given by mutation paths. 
\end{defn}

\begin{rem}
Following the notion of positive space of \cite{FockGonch:2}, we should really be considering our positive spaces $\mathcal{A}_Q$ and $\mathcal{X}_Q$ as being functors from the groupoid $\mathcal{G}_Q$ to the category of positive varieties. 
\end{rem}

In order to give a concrete description of this groupoid, we mark each isomorphism class with a particular representative. The maps from two quiver isomorphism classes representatives, $Q_1$ and $Q_2$, are given by elements $\{P,\sigma\}$, where $P$ is a mutation path from $Q_1$ to a quiver isomorphic to $Q_2$ and $\sigma$ is an isomorphism from $Q_2$ to $P(Q_1)$. As before, any map that induces the identity on all of the cluster $\mathcal{X}$ coordinates is considered to be an identity map in $\mathcal G_Q$.

With this definition we may see that $\Gamma_Q$ is the group associated with $\mathcal{G}_Q$ and that the algebraic structures of $\mathcal{G}_Q$ and $\Gamma_Q$ only depend on the mutation class of $Q$.

\begin{exmp}

\begin{figure}[hb]
    \centering
    \begin{subfigure}[]{0.3\textwidth}
        \begin{tikzcd}
             & 1 \arrow[rd] &              \\
             3 \arrow[ru] &              & 2 \arrow[ll]
        \end{tikzcd}
        \caption{$Q_1$}
    \end{subfigure}   
    \begin{subfigure}[]{0.3\textwidth}
        \begin{tikzcd}
          & 1 \arrow[ld] &              \\
            3 &              & 2 \arrow[lu]
        \end{tikzcd}
          \caption{$Q_2$}
    \end{subfigure}
    \\
    \begin{subfigure}[]{0.3\textwidth}
       \begin{tikzcd}
              & 1 \arrow[ld] \arrow[rd] &   \\
           3 &                         & 2
        \end{tikzcd}
          \caption{$Q_3$}
    \end{subfigure}    
    \begin{subfigure}[]{0.3\textwidth}
        \begin{tikzcd}
             & 1 &              \\
              3 \arrow[ru] &   & 2 \arrow[lu]
        \end{tikzcd}
          \caption{$Q_4$}
    \end{subfigure}

    \caption{Choices of quiver isomorphism classes in the mutation class of a quiver of type $A_3$.}
    \label{fig:A3QuiverReps}
\end{figure}
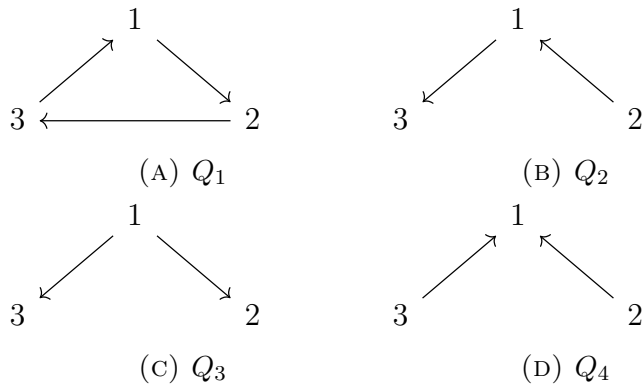

If we take a quiver of type $A_3$ as our initial seed, it is easy to compute the cluster modular groupoid and cluster modular group. 
We can do this by taking a representative for each of the quivers in the mutation class of our seed and then writing all of the possible pairs of mutation paths and quiver isomorphisms between them. A choice of possible representatives for the quiver isomorphism classes is shown in figure \ref{fig:A3QuiverReps} and a diagram of the groupoid is shown in figure \ref{fig:A3groupoid}. It is not too difficult to compute the cluster modular group by checking that every path from $Q_1$ to itself is generated by the two paths shown in the figure. These two elements commute and the top has order 2 and the bottom has order 3. Thus the cluster modular group is $\mathbb{Z}/6\mathbb{Z}$

\begin{figure}[ht]
    \centering
    \begin{tikzcd}[sep = huge]
                                                                                                            &                                                                           & Q_3 \arrow[dd,leftrightarrow, "{\{1,()\}}" description] \arrow[ld, "{\{3,(23)\}}" description, bend right=49] \\
Q_1 \arrow[r,leftrightarrow, "{\{1,()\}}" description] \arrow["{\{<>,(123)\}}",swap , loop, distance=4em, in=305, out=235] \arrow["{\{1231,(23)\}}" , loop, distance=4em, in=55, out=125] & Q_2 \arrow[ru,leftrightarrow, "{\{2,()\}}" description] \arrow[rd,leftrightarrow, "{\{3,()\}}" description] &                                                                                            \\
                                                                                                            &                                                                           & Q_4 \arrow["{\{<>,(23)\}}", swap,loop right, in=305, out=235,distance = 4em]                  
    \end{tikzcd}
    \caption{The cluster modular groupoid of the $A_3$ cluster ensemble. A selection of non-identity maps are shown.}
    \label{fig:A3groupoid}
\end{figure}
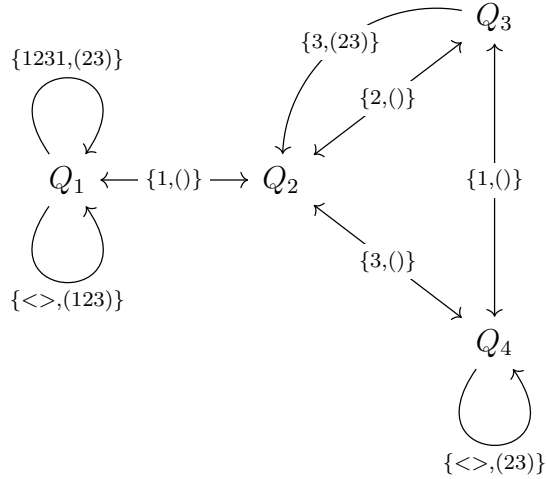

\end{exmp}

\subsection{Action of the cluster modular group on functions}

We write, by abuse of notation, the field of rational functions on a positive space, $\mathbb{R}(\mathcal{A}_Q)$, to mean $\mathbb{R}(\mathcal{A}_{|Q|})$ and similarly for $\mathbb{R}(\mathcal{X}_Q)$.

Again fix a starting seed with quiver $Q$. We can now define an action of $\mathcal{G}_Q$ on the set of rational functions on the $\mathcal A_Q$ or $\mathcal X_Q$ space. The cluster coordinates of our seed give a basis of the field of rational functions on the $\mathcal A_Q$ and $\mathcal X_Q$ spaces. 

Given $f \in \mathbb{R}(\mathcal A_{Q})$ we can define $\gamma(f)(a_1,\dots,a_n) = \sigma^{}(P^{*}(f))$,  where $P$ acts by pullback along $P: \mathcal{A}_{|Q|} \rightarrow \mathcal{A}_{|P(Q)|}$, and $\sigma$ acts as a map from the cluster variables on the seed $P(Q)$ back to the initial seed $Q$. We define an action on elements of $\mathbb{R}(\mathcal X_{Q})$ in an analogous way. 

\begin{defn}
We call the action of $\mathcal{G}$ on functions \emph{functional exchanges} and we call the set $\Gamma_Q(f) $ the \emph{exchange class} of $f$. The set $\mathcal{G}(f)$ provides a set of exchange classes for each quiver isomorphism class in the mutation class of $Q$. 
\end{defn}

\begin{exmp}
\label{exmp:a3function}
Let $Q = Q_1$ of figure \ref{fig:A3QuiverReps}. Let $f \in \mathbb{R}(\mathcal{A}_Q)$ be given by
\begin{equation}
    f(a_1,a_2,a_3) = \frac{a_1+a_2+a_3}{a_1a_2a_3}.
\end{equation}
Then we may compute that $\{1231,(23)\}(f) = f^{-1}$.
Let us write this out in detail. We write $a_i,b_i,c_i,d_i,e_i$ for the variables on each of the five seeds found along the path. The important exchange relations are
\begin{align}
    \mu_1^*(b_1) &= \frac{a_2+a_3}{a_1} &
    \mu_2^*(c_1) &= \frac{1+b_1}{b_2} \\
    \mu_3^*(d_3) &= \frac{1+c_1}{c_2} &
    \mu_1^*(e_1) &= \frac{d_2+d_3}{d_1}
\end{align}
and all other exchange relations are trivial. 

The isomorphism $(23)$ provides the map $\{a_1,a_3,a_2\} \xrightarrow{(23)} \{e_1,e_2,e_3\}$ and the mutation path acts on $(23)(f)$ by pullback. Writing the value of $f$ by abuse of notation, we compute that
\begin{align*}
   \{1231,(23)\}(f) = & \mu_1 \circ \mu_2 \circ \mu_3  \circ \mu_1 \left((23)(\frac{a_1+a_2+a_3}{a_1a_2a_3})\right) =  \\
    & \mu_1 \circ \mu_2 \circ \mu_3  \circ \mu_1 \left (\frac{e_1+e_3+e_2}{e_1e_3e_2} \right) = 
     \mu_1 \circ \mu_2 \circ \mu_3 \left (\frac{d_1+1}{d_2d_3}\right) = \\
    & \mu_1 \circ \mu_2 \left (\frac{c_2}{c_3}\right) = 
     \mu_1\left(\frac{b_2b_3}{b_1+1}\right) = \frac{a_1a_2a_3}{a_1+a_2+a_3} =  f^{-1}.
\end{align*}

We clearly have that $\{<>,(123)\}(f) = f$. Therefore, the cluster modular group $\Gamma_Q = \mathbb{Z}/6\mathbb{Z}$ acts on $f$ via the quotient group $\mathbb{Z}/2\mathbb{Z}$.

\end{exmp}

\subsection{Invariants of the cluster modular group}

\begin{defn} 
Let $\Gamma_\circ \subset \Gamma_Q$ be a subgroup. An invariant function for $\Gamma_\circ$ on  $\mathcal A_Q$  or $\mathcal X_Q$ is an element of the field of invariants $\mathbb{R}(\mathcal A_Q)^{\Gamma_\circ}$ or $\mathbb{R}(\mathcal X_Q)^{\Gamma_\circ}$ respectively. 
\end{defn}

\section{Examples}\label{sec:examples}

\begin{exmp}
\label{exmp:a11}

Let $Q$ be a quiver of type $\widetilde{A}_1$ shown in figure \ref{fig:a1affine}. The cluster modular group, $\Gamma_Q$, is isomorphic to $\mathbb{Z}$ generated by $\{1,(12)\}$. It is not too difficult to write an invariant for $\Gamma_Q$ on $\mathcal A_Q$. The function 
\begin{equation}
F(a_1,a_2)= \dfrac{1+a_1^2+a_2^2}{a_1a_2}
\end{equation}
is invariant. We also have an invariant for this group on $\mathcal X_Q$ : 
\begin{equation}
G(x_1,x_2)= \dfrac{(x_2(x_1+1)+1)^2}{x_1x_2}
\end{equation}
It is easy to check that $\rho^*(G)=F^2$

We can use this invariant to study the limiting behavior of the cluster variables for each of the $\mathcal A$ and $\mathcal X$ spaces under the action of $\Gamma$. We find that the $\mathcal A$ coordinates behave like a geometric sequence after many mutations. Let $(a_1,a_2)$ be the initial cluster and $(a_2,a_3),(a_3,a_4), \dots (a_{n},a_{n+1})$ be the subsequent clusters obtained after applying  $\gamma = \{1,(12)\}$ $n-1$ times. Then we have that $ F(a_1,a_2) = F(a_n,a_{a+1})$. Call this expression $F$ for simplicity. Using the exchange relation 
\begin{equation}
    a_{n-1}=\frac{a_n^2+1}{a_{n+1}}
\end{equation}
 we can see using the fact that $  F(a_n,a_{n+1}) = F $ that 
 \begin{equation}
    a_{n-1}+a_{n+1} = Fa_n 
 \end{equation}
 Thus the sequence of ${a_i}$ satisfies a linear recurrence and we must have that $\displaystyle{\lim_{i\rightarrow \infty}} \frac{a_{i+1}}{a_i} = \lambda$ for some multiplier $\lambda$. It is easy to compute that
 \begin{equation}
     \lambda = \frac{F+\sqrt{F^2-4}}{2} = \exp(\arccosh(F/2))
 \end{equation}
 
This expression for $\lambda$ is vital for the analysis of the limiting behavior of the Markov numbers done in \cite{Zagier:Markoff}. We will see in example \ref{exmp:hypa11} that the appearance of $\arccosh$ is related to the hyperbolic geometry of an annulus.  

\begin{figure}[t]
    \centering
    \begin{tikzcd}
1 \arrow[r, Rightarrow] & 2
\end{tikzcd}
    \caption{Quiver of type $\widetilde{A}_1$. }
    \label{fig:a1affine}
\end{figure}
 
\end{exmp}

\begin{exmp} 
\label{exmp:markov}
Let $(\mathcal A_Q,\mathcal X_Q)$ be the cluster ensemble with trivial coefficients associated with the Markov quiver, $Q$, of figure \ref{fig:markovQuiver}. The function 
\begin{equation}
     F(a_1,a_2,a_3)= \dfrac{a_1^2+a_2^2+a_3^2}{a_1a_2a_3} 
\end{equation}
is an invariant function for all of $\Gamma_Q =  \PSL(2,\mathbb{Z})$ on $\mathcal{A}_Q$. The function 
\begin{equation}
    G(x_1,x_2,x_3)= x_1x_2x_3
\end{equation}
is an invariant for $\Gamma_Q$ on $\mathcal X_Q$. 

The function $F$ encapsulates the Diophantine properties of the Markov numbers. The Markov numbers are generated by evaluating the $\mathcal{A}$ coordinates at $(1,1,1)$. We write these as triples of integers $(x,y,z)$. Since $ F(1,1,1) = 3  $, we have that the Markov numbers all satisfy $ x^2+y^2+z^2 - 3xyz = 0 $. 

If we freeze any one of the nodes in $Q$, then remaining mutable portion is a quiver of type $\widetilde{A}_1$. We can use the invariant in the exact same way as example \ref{exmp:a11} to study the limiting behavior of mutations on this modified quiver. In this way we recover similar analysis of \cite{Zagier:Markoff}.

\end{exmp}
There are 2 other quivers with  similar properties to the Markov quiver. All 3 of these quivers have 3 nodes and are associated doubly extended Dynkin diagrams\footnote{ See \cite{Saito:1} for background on the notation used in this paper. We only use this naming convention to emphasise a deeper connection between these examples.}. We will examine the invariant functions associated to cluster ensembles with trivial coefficients of these types.
\begin{exmp}
\label{exmp:bc21}
Let $Q$ be a quiver of type $ BC_1^{(2,1)} $
Then the function 
\begin{equation}
F(a_1,a_2,a_3)= \dfrac{a_1^4+(a_2+a_3)^2}{a_1^2a_2a_3} 
\end{equation}
is an element of $\mathbb{R}(\mathcal A_Q)^{\Gamma_Q}$.

\end{exmp}
\begin{figure}[t]
\begin{center}
\begin{subfigure}{0.29\textwidth}
\centering 
\begin{tikzcd}[ampersand replacement=\&,sep=scriptsize]
                               \& 1 \arrow[rdd, Rightarrow] \&                               \\
                               \&                                \&                               \\
3 \arrow[ruu, Rightarrow] \&                                \& 2 \arrow[ll, Rightarrow]
\end{tikzcd}
\caption{The Markov quiver, $A_1^{(1,1)}$.}
\label{fig:markovQuiver}
\end{subfigure}
\begin{subfigure}{0.29\textwidth}
\centering
\begin{tikzcd}[ampersand replacement=\&,sep=scriptsize]
                               \& 1^{[4]} \arrow[rdd, "{4,1}" description] \&                               \\
                               \&                                \&                               \\
3 \arrow[ruu, "{1,4}" description] \&                                \& 2 \arrow[ll, Rightarrow]
\end{tikzcd}
\caption{$ BC_1^{(2,1)} $}
\end{subfigure}
\begin{subfigure}{0.29\textwidth}
\centering
\begin{tikzcd}[ampersand replacement=\&,sep=scriptsize]
                               \& 1 \arrow[rdd, "{1,4}" description] \&                               \\
                               \&                                \&                               \\
3^{[4]} \arrow[ruu, "{4,1}" description] \&                                \& 2^{[4]} \arrow[ll, Rightarrow]
\end{tikzcd}
\caption{$ BC_1^{(2,4)} $}
\end{subfigure}
\end{center}
\caption{Quivers associated with doubly extended Dynkin diagrams with 3 nodes.}
\label{fig:rank1EllipticQuivers}
\end{figure}
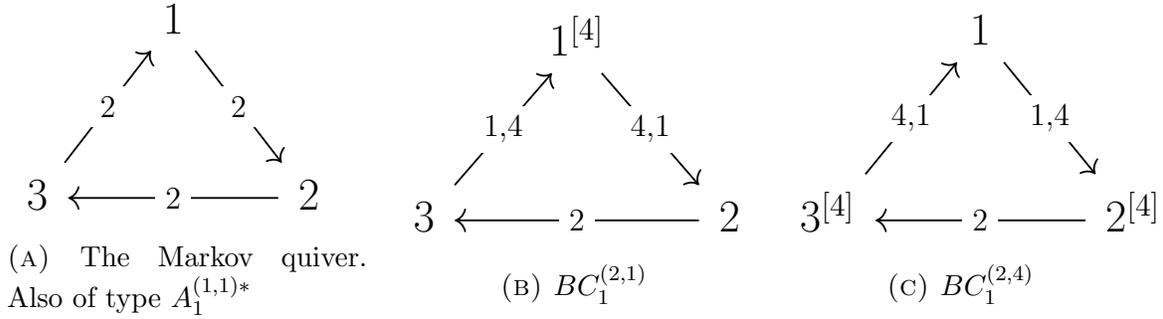

\begin{exmp} \label{exmp:bc24}
Let $Q$ be a quiver of type $BC_1^{(2,4)}$
The function 
\begin{equation}
F(a_1,a_2,a_3)= \dfrac{a_1^2+2a_1(a_2^2+a_3^2)+a_2^4+a_3^4}{a_1a_2^2a_3^2} 
\end{equation}
is an element of $\mathbb{R}(\mathcal A_Q)^{\Gamma_Q}$. This function and its Diophantine properties were studied in \cite{Lampe}.

\end{exmp}

\begin{exmp} \label{exmp:somos}

\begin{figure}[b]
    \centering
    \begin{subfigure}[t!]{0.3\textwidth}
    \centering
    \begin{tikzcd}[sep = 4em]
1 \arrow[rd, Rightarrow] & 2 \arrow[l] \arrow[ld, Rightarrow] \\
4 \arrow[u] \arrow[r]    & 3 \arrow[u, "3" description]      
\end{tikzcd}
    
    \caption{$Q_{s4}$}
    \end{subfigure}
    \begin{subfigure}[t!]{0.3\textwidth}
    \centering
    \begin{tikzcd}
                          & 3 \arrow[rdd] \arrow[ld, Rightarrow] &                          \\
2 \arrow[d] \arrow[rrd]   &                                      & 4 \arrow[lu, Rightarrow] \\
1 \arrow[ruu] \arrow[rru] &                                      & 5 \arrow[ll] \arrow[u]  
\end{tikzcd}
     \caption{$Q_{s5}$}
    \end{subfigure}
    \begin{subfigure}{0.3\textwidth}
    \centering
    \begin{tikzcd}[column sep = 1.4em]
                                     & 3 \arrow[r] \arrow[rrd, Rightarrow] & 4 \arrow[rd] \arrow[ldd, Rightarrow] &                                      \\
2 \arrow[ru] \arrow[rru, Rightarrow] &                                     &                                      & 5 \arrow[ld] \arrow[lll, Rightarrow] \\
                                     & 1 \arrow[lu] \arrow[r]              & 6 \arrow[luu, Rightarrow]            &                                     
\end{tikzcd}
    \caption{$Q_{s6}$}
    \end{subfigure}
    \label{fig:somos}
    \caption{Quivers for the Somos 4, 5 and 6 sequences.}
\end{figure}
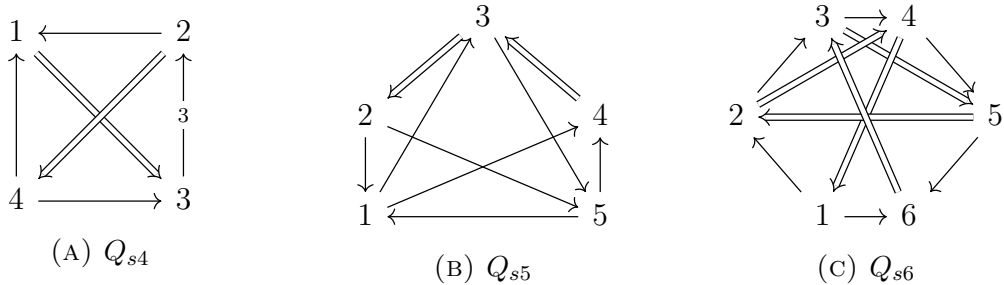

The Somos 4, 5 and 6 sequences can be associated to cluster algebras in a simple way\footnote{ This Somos 6 sequence is not the most famous instance, see \cite{lampe:descrete_integrability} for a discussion of the example shown here.}. If we look at the variables on the $\mathcal A$ space generated by starting with quivers $Q_{s4}$ , $Q_{s5}$ or $Q_{s6}$ and following the mutation paths $\gamma_4= \{1,(1234)\}$ , $\gamma_5 = \{1,(12345)\}$ or $\gamma_6 = \{1,(123456)\}$, we obtain the respective Somos sequences by evaluating all of the initial cluster variables at 1. There are invariant functions on each of these cluster ensembles for these mutation paths. The functions 
\begin{equation}
    F_4(a_1,a_2,a_3,a_4) = \dfrac{a_1^2a_4^2+a_1a_3^3+a_4a_2^3+a_2^2a_3^2}{a_1a_2a_3a_4}
\end{equation}
\begin{equation}
    F_5(a_1,a_2,a_3,a_4,a_5) = \dfrac{a_1^2a_4^2a_5+a_1a_2^2a_5^2+a_1a_3^2a_4^2+a_2^2a_3^2a_5+a_2a_3^3a_4}{a_1a_2a_3a_4a_5}
\end{equation}
and
\begin{equation}
\begin{split}
    & F_6(a_1,a_2,a_3,a_4,a_5,a_6) =  \\ &\frac{a_1^2a_2a_5a_6^2+a_1^2a_4a_5^3+a_2^3a_3a_6^2+a_1a_3^2a_4a_5^2+a_2^2a_3a_4^2a_6}{a_1a_2a_3a_4a_5a_6}  \\
    &+\frac{a_1a_3a_4^3a_5+a_2a_3^3a_4a_6+a_3^3a_4^3}{a_1a_2a_3a_4a_5a_6}
\end{split}
\end{equation}

are elements of $\mathbb{R}(\mathcal A_{Q_{s4}})^{<\gamma_4>}$ , $\mathbb{R}(\mathcal A_{Q_{s5}})^{<\gamma_5>}$ and $\mathbb{R}(\mathcal A_{Q_{s6}})^{<\gamma_6>}$ respectively.

The functions $F_4$ and $F_5$ have appeared in \cite{Somos} in relation to the number theoretic properties of the Somos sequences e.g. one may compute the $j-$invariant of an elliptic curve associated with the sequences using the values of these functions. The function $F_6$ has not been referenced to our knowledge, and we strongly suspect that there are similar functions for more general Somos-like integer sequences.  

\end{exmp}

In all of the above examples, the associated quivers all had only one relevant quiver isomorphism class. This meant that the cluster modular group could be generated by paths with only one mutation and it is relatively simple to check that a function is invariant. The next examples show that we can use our generalized notions to discuss invariants for more complicated subgroups of the cluster modular group. 

\begin{exmp}

\begin{figure}[t]
    \centering
    \begin{tikzcd}
             & 3 \arrow[r]  & 4 \arrow[rd] &                      \\
    2 \arrow[ru] &              &              & 5 \arrow[ld, dashed] \\
             & 1 \arrow[lu] & n \arrow[l]  &                     
    \end{tikzcd}
    \caption{$Q_{n-cycle}$. This quiver is simply an $n$ cycle of nodes and arrows.}
    \label{fig:D_Ncycle}
\end{figure}
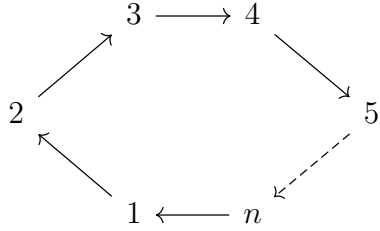

The mutation class of a quiver of type $D_n$ has an element that looks like a directed $n$-cycle, as shown in figure \ref{fig:D_Ncycle}. Let $Q$ be this quiver. The cluster modular group for $Q$ is $\mathbb{Z}/n\mathbb{Z}\times\mathbb{Z}/2\mathbb{Z}$ generated by $r$ and $t$ in each of the factors\footnote{This is not quite true when $n=4$, but the given function is still an invariant.}. The function
\begin{equation}
    F(a_1,a_2,\dots,a_n)= \dfrac{1}{a_1a_2}+\dfrac{1}{a_2a_3}+\dots+\dfrac{1}{a_na_1}
\end{equation}
is an element of $\mathbb{R}(\mathcal{A}_Q)^{<r>}$ and satisfies $t(F)=F^{-1}$. This is a generalization of example \ref{exmp:a3function}. 

\end{exmp}

We will examine some examples of invariants for ensembles associated to more doubly extended Dynkin diagrams with more than 3 nodes. 

\begin{exmp}
\label{G233}

\begin{figure}[b]
    \centering

    \begin{tikzcd}
                                  & 1^{[3]} \arrow[dd, "2" description]                &              \\
    4 \arrow[ru, "{1,3}" description] &                                              & 3^{[3]} \arrow[lu] \\
                                  & 2^{[3]} \arrow[ru] \arrow[lu, "{3,1}" description] &             
\end{tikzcd}
    \caption{Quiver of type $G_2^{(3,3)}$. }
    \label{fig:g233}
\end{figure}
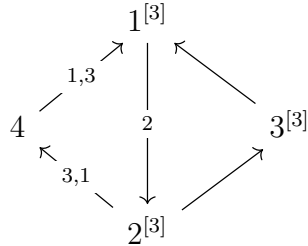

A quiver of type $G_2^{(3,3)}$ has two quiver isomorphism classes. Let $Q$ be the isomorphism class shown in figure \ref{fig:g233}. We can see that the element $\gamma = \{1,(12)\}$ is in $\Gamma_Q$. It is not hard to compute that 
\begin{equation}
    1 \rightarrow N(\gamma) \rightarrow \Gamma_Q \rightarrow D_{12} \rightarrow 1
\end{equation}
where $N(\gamma)$ is the normalizer of the element $\gamma$ and $D_{12}$ is the dihedral group with 12 elements. 
The functions 
\begin{equation}
    F_1(a_1,a_2,a_3,a_4) = \dfrac{(a_1^3+a_2^3)(a_1+a_2)+a_3a_4(2a_1^2+a_1a_2+2a_2^2)+a_3^2a_4^2}{a_1^2a_2^2a_3a_4}
\end{equation}
and
\begin{equation}
    F_2(a_1,a_2,a_3,a_4) = \dfrac{(a_1+a_2)^2+a_3a_4}{a_1a_2a_3^2}
\end{equation}
are elements of $\mathbb{R}(\mathcal{A}_{Q})^{N(\gamma)}$.

\end{exmp}

\begin{exmp}
\label{exmp:d411:1}

A $ D_4^{(1,1)} $ type quiver has 4 mutation classes, call them $ Q_1, \dots, Q_4 $, see figure \ref{fig:D4-11_quivers}. Let $Q_4$ be our seed. The cluster modular group can be written as an extension 
\begin{equation}
    1 \rightarrow \mathbb{Z}*\mathbb{Z} \rightarrow \Gamma_{Q_4} \rightarrow \Aut(F_4)^+ \rightarrow 1
\end{equation}
where $W(F_4)^+$ is the orientation preserving part of the automorphism group of the $F_4$ root system of order 1152\footnote{If we allowed arrow reversing quiver isomorphisms, we would get all of $\Aut(F_4)$, see \cite{kauf:Modular_groups}.}. We will see in example \ref{exmp:4punctured} why $\mathbb{Z}*\mathbb{Z}$ is a normal subgroup. $\Aut(F_4)^+$ is generated by the elements of $\Aut(Q_4)$ and the path $\{1231,(23)\}$.

An element of $\mathbb{R}(\mathcal{A}_{Q_4})^{\mathbb{Z}*\mathbb{Z}}$ is 
\begin{equation}
F(a_1,a_2,a_3,a_4,a_5,a_6)= \dfrac{(a_1a_4+a_2a_5+a_3a_6)^2}{a_1a_2a_3a_4a_5a_6} 
\end{equation}

This is not the only function in its exchange class, that is to say that $W(F_4)^+$ acts non trivially on this function. After applying the element $\{1231,(23)\}$, we obtain the function
\begin{equation}
 F_{456}(a_1,a_2,a_3,a_4,a_5,a_6)= \dfrac{a_1a_4+a_2a_5+a_3a_6}{a_4a_5a_6} 
\end{equation}

There are 24 different functions in the exchange class of $F$. They are 
\begin{equation}
    F,\quad F^{-1}, \quad F_{456},F_{456}^{-1}, \quad \dfrac{a_1}{a_4}
\end{equation}
along with each of their images under the automorphism group of $Q_4$.

There is a relationship with these functions and the invariant of the Markov quiver. If we fold (see \cite{FST:finite_mutation_via_unfoldings} section 4 for a background on folding) $Q_4$ by associating nodes 1 and 4, 2 and 5, 3 and 6, we obtain a  quiver of type $A_1^{(1,1)}$.  Under this folding, each of these functions becomes an invariant function for a cluster algebra of type $A_1^{(1,1)}$. In other words, if we set $a_1=a_4 ,a_2=a_5, a_3=a_6$ then our functions become either $ 1, F'$ or $F'^2$, where $F'$ is the invariant of example \ref{exmp:markov}. 

There is also a relationship between these invariants and the invariants of the $BC_1^{(2,1)}$ ensemble of example \ref{exmp:bc21}. We can fold $Q_1$ to obtain a quiver of type  $BC_1^{(2,1)}$ by associating nodes 2, 3, 5, and 6. We can find invariants for $\mathcal{A}_{Q_1}$ by pulling our set of invariants on $\mathcal{A}_{Q_4}$ back along a mutation path between these two quivers. In doing so, we obtain the functions
\begin{equation}
    F_a = \dfrac{(a_1+a_4)^2+a_2a_3a_5a_6}{a_1a_4a_2a_5}, \quad
    F_a^{-1}, \quad
    F_b= \dfrac{a_2}{a_3} 
\end{equation}
and each of their images under the action of $\Aut(Q_1)$, as elements of $\mathbb{R}(\mathcal{A}_{Q_1})^{\mathbb{Z}*\mathbb{Z}}$. Each of these clearly folds to an invariant of the $BC_1^{(2,1)}$ algebra by setting $a_2=a_3=a_5=a_6$.

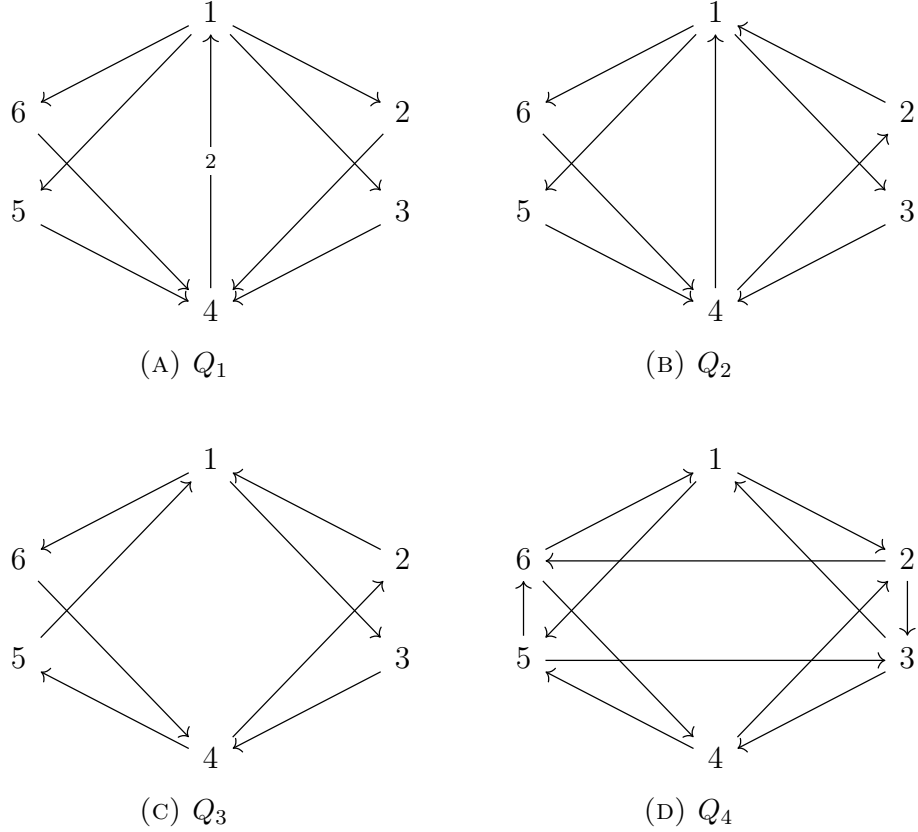
\begin{figure}[t]
\centering
\begin{subfigure}{0.49\textwidth}
\centering
\begin{tikzcd}[ampersand replacement=\&]
               \&  \& 1 \arrow[rrd] \arrow[rrdd] \arrow[lld] \arrow[lldd] \&  \&                \\
6 \arrow[rrdd] \&  \&                                                     \&  \& 2 \arrow[lldd] \\
5 \arrow[rrd]  \&  \&                                                     \&  \& 3 \arrow[lld]  \\
               \&  \& 4 \arrow[uuu, "2" description]                      \&  \&               
\end{tikzcd}
\caption{$Q_1$}
\end{subfigure}
\begin{subfigure}{0.49\textwidth}
\centering
\begin{tikzcd}[ampersand replacement=\&]
               \&  \& 1 \arrow[rrdd] \arrow[lldd] \arrow[lld] \&  \&               \\
6 \arrow[rrdd] \&  \&                                         \&  \& 2 \arrow[llu] \\
5 \arrow[rrd]  \&  \&                                         \&  \& 3 \arrow[lld] \\
               \&  \& 4 \arrow[uuu] \arrow[rruu]              \&  \&              
\end{tikzcd}
\caption{$Q_2$}
\end{subfigure}

\begin{subfigure}{0.49\textwidth}
\centering
\begin{tikzcd}[ampersand replacement=\&]
               \&  \& 1 \arrow[rrdd] \arrow[lld] \&  \&               \\
6 \arrow[rrdd] \&  \&                            \&  \& 2 \arrow[llu] \\
5 \arrow[rruu] \&  \&                            \&  \& 3 \arrow[lld] \\
               \&  \& 4 \arrow[rruu] \arrow[llu] \&  \&              
\end{tikzcd}
\caption{$Q_3$}
\end{subfigure}
\begin{subfigure}{0.49\textwidth}
\centering
\begin{tikzcd}[ampersand replacement=\&]
                           \&  \& 1 \arrow[rrd] \arrow[lldd] \&  \&                            \\
6 \arrow[rrdd] \arrow[rru] \&  \&                            \&  \& 2 \arrow[d] \arrow[llll]   \\
5 \arrow[u] \arrow[rrrr]   \&  \&                            \&  \& 3 \arrow[lld] \arrow[lluu] \\
                           \&  \& 4 \arrow[rruu] \arrow[llu] \&  \&                           
\end{tikzcd}
\caption{$Q_4$}
\end{subfigure}
\caption{Quivers of type $ D_4^{(1,1)} $.}
\label{fig:D4-11_quivers}
\end{figure}

\end{exmp}

\begin{exmp}
\label{d411:2}
The Weyl group, $W(D_4) = \mathbb{Z}_2^3 \rtimes S_4$, is a normal subgroup of the modular group of the $D_4^{(1,1)}$ cluster ensemble. For a simple presentation for this, take $Q_1$ as our initial quiver and look at the subgroup generated by the $\Aut(Q_1)=S_4$ and the path $\{214,(142)\}*\{314,(143)\}^{-1}$. This generates $W(D_4)$ as a subgroup of $\mathbb{Z}_2^4 \rtimes S_4$ generated by the elements of $S_4$ and the element $((1,-1,0,0),id)$. We have an exact sequence:
\begin{equation}
    1 \rightarrow W(D_4) \rightarrow \Gamma \rightarrow \PSL(2,\mathbb{Z}) \rightarrow 1
\end{equation}
Consider $Q_4$ as our seed. The function 
\begin{equation}
    F(a_1,a_2,a_3,a_4,a_5,a_6)= \dfrac{(a_1a_4+a_2a_5+a_3a_6)^3}{a_1a_2a_3a_4a_5a_6}
\end{equation}
is an element of $\mathbb{R}(\mathcal A_{Q_4})^{W(D_4)}$.  
\end{exmp}

\section{Basic properties}\label{sec:structures}
We will compile some elementary facts about invariants that will be useful later. 

\subsection{Trivial Invariants}

Up until now, we have not considered any quivers that have frozen variables, and the associated $\mathcal{A}$ spaces were free of coefficients.  

Since it respects mutations, the map $\rho: \mathcal{A}_Q \rightarrow \mathcal{X}_Q$ gives a map 
\begin{equation}
    \rho^*: \mathbb{R}(\mathcal{X}_Q)^{\Gamma_\circ} \cup \{\infty\} \rightarrow \mathbb{R}(\mathcal{A}_Q)^{\Gamma_\circ} \cup \{\infty\}
\end{equation}
by pullback. This is very useful for studying the invariants on the $\mathcal A$ space when there are nontrivial coefficients. However, since $\rho$ is not surjective or injective in general, we cannot say much about the image of the set of $\mathcal X$ invariants under $\rho^*$. 

\begin{defn}
A rational function in the coefficients of $\mathcal A_Q$ is called a \emph{trivial} invariant.
\end{defn}

These functions are clearly unchanged by the action of cluster modular group elements.

We can understand what functions pull-back to trivial $\mathcal A$ invariants via $\rho^*$ as follows: any vector, $v= (v_1,\dots,v_n)$  of the kernel of the exchange matrix gives such a function since $\rho^*(\displaystyle{\prod{x_i^{v_i}}})$ is a function just of frozen variables, see section 2.3 of \cite{FockGonch:2}. These functions are called ``Casimir'' functions.

\begin{defn}
A $\mathcal X$ invariant that is also a Casimir function will be called a Casimir $\mathcal X$ invariant. 
\end{defn}

\begin{exmp}
The type $A_n$ cluster ensembles have an interesting trivial $\mathcal X$ invariant when $n$ is odd.  The cluster modular group of the type $A_n$ cluster ensemble is $\mathbb{Z}/(n+3)\mathbb{Z}$, generated by $\gamma$. Take $Q$ to be the quiver in figure \ref{fig:An_Odd} and let $G(x_1,\dots,x_{2n+1})= x_1x_3 \dots x_{2n+1}$. Then $G$ is a Casimir invariant in $\mathbb{R}({\mathcal X}_Q)^{\mathbb{Z}/m\mathbb{Z}}$, where $m=\frac{n+3}{2}$. $G$ satisfies $\gamma(G) = G^{-1}$.

\begin{figure}[hb]
    \centering
    \begin{tikzcd}
        1 \arrow[r] & 2 \arrow[r] & ... \arrow[r] & 2n+1
    \end{tikzcd}
    \caption{Quiver of type $A_{2n+1}$.}
    \label{fig:An_Odd}
\end{figure}

\end{exmp}

\subsection{Invariants of Subensembles} 
The invariants of the ${\mathcal A}$ space and $\mathcal X$ space of an ensemble respect inclusions of ``subensembles'' in different ways.

\begin{defn}

Let $R^\mu \subset Q^\mu$ be a subquiver of the mutable portion of $Q$. Let $R$ be the quiver obtained from $Q$ by freezing all the nodes in $N(Q)-N(R^\mu)$. We consider $({\mathcal A}_R,{\mathcal X}_R) $  to be a \emph{subensemble} of $({\mathcal A}_Q,{\mathcal X}_Q)$. 

\end{defn}

Let $\gamma_R = \{P_R,\sigma_R\}$ be an element of the cluster modular group of $R$ that extends to an element of the modular group of $Q$, i.e. there exists $\gamma_Q = \{P_R,\sigma_Q\}$ where $\sigma_Q|_{R}=\sigma_R$.
We can now state the following pair of theorems:

\begin{thm} \label{thm:adims}
$\mathbb{R}(\mathcal{A}_Q)^{<\gamma_Q>}$ is exactly $\mathbb{R}(\mathcal{A}_R)^{<\gamma_R>}$. Furthermore, evaluating the variables associated to the nodes in $N(Q)-N(R^\mu)$ at 1 gives a map of sets $\mathbb{R}(\mathcal{A}_Q)^{<\gamma_Q>}\cup\{\infty\} \rightarrow \mathbb{R}(\mathcal{A}_{R^\mu})^{<\gamma_R>}\cup\{\infty\}$
\end{thm}

\begin{thm} \label{thm:xdims}
We have a natural inclusion $\mathbb{R}(\mathcal{X}_R)^{<\gamma_R>} \subset \mathbb{R}(\mathcal{X}_Q)^{<\gamma_Q>}$. Furthermore, evaluating the variables associated to nodes in $N(Q^\mu)-N(R^\mu)$ at 0 gives a surjection of sets $\mathbb{R}(\mathcal{X}_Q)^{<\gamma_Q>}\cup\{\infty\} \rightarrow \mathbb{R}(\mathcal{X}_R)^{<\gamma_R>} \cup \{\infty\}$
\end{thm}

\begin{proof}

The first theorem follows from the fact that the mutation rule on the $\mathcal A$ space only changes the variable being mutated. The first part of the second theorem follows since the $\mathcal X$ coordinates are unchanged when frozen variables are added. The only thing left to prove is that evaluating the nodes in $N(Q^\mu)-N(R^\mu)$ at 0 gives a map from $\mathbb{R}(\mathcal{X}_Q)^{<\gamma_Q>}$ to $\mathbb{R}(\mathcal{X}_R)^{<\gamma_R>}$. The surjection is obvious because of the inclusion of rings.

Let $r = \#N(R^\mu)$ and renumber the nodes of $Q$ so that the first $r$ nodes are the nodes in $N(R^\mu)$. Let $G_i(x_1,\dots,x_n)=x_i$ and let $j= \sigma(i)$. Then for any $i > r$, we have
\begin{equation}
    \gamma_Q(G_i)=x_{j}h_j(x_1,\dots,x_r)
\end{equation}, where $h_j$ is some rational function that does not depend on any $x_i, i>r$. This follows since we never mutate at node $i$ in $\gamma_Q$, and the $\mathcal X$ mutation rule only multiplies coordinates by functions of those which are mutated.

Let $K = \mathbb{R}[x_{r+1},\dots,x_n]$ and $L = \mathbb{R}(x_1,\dots,x_r) = \mathbb{R}(\mathcal X_R)$. Then, $\gamma_Q$ acts on any monomial term in $K$ by multiplication by an element of $L$.
Now we write any function, 
\begin{equation}
    g = p/q \in \mathbb{R}(\mathcal X_Q) = \mathbb{R}(x_1,\dots,x_n),
\end{equation} 
as a ratio of two polynomials in $K$ with coefficients in $L$. Let $p_0,q_0 \in L$ be the constant terms of $p$ and $q$. Now we may see that if $g$ is invariant then we have 
\begin{equation}
    p\gamma_Q(q) = q\gamma_Q(p)
\end{equation}
which implies by comparing constant terms that
\begin{equation}
    p_0\gamma_Q(q_0) = q_0\gamma_Q(p_0)
\end{equation}
and so we have that $\frac{p_0}{q_0} \in L$ is invariant.

\end{proof}

\begin{cor} \label{cor:xdims}
 The dimension of the field extension $\mathbb{R}(\mathcal{X}_Q)^{<\gamma_Q>} / \mathbb{R}(\mathcal{X}_R)^{<\gamma_R>} \leq n-r$. 
\end{cor}

\begin{proof}

This follows from the proof of the previous theorem. Since $\gamma_Q$ acts on each monomial in $K$ independently by multiplication by elements in $L$, any basis of $\mathbb{R}(\mathcal{X}_Q)^{<\gamma_Q> }/ \mathbb{R}(\mathcal{X}_R)^{<\gamma_R>}$ can be written as a basis of monomials in $K$ with coefficients in $L$. Now we may see that any particular monomial in $K$ can only be associated with one independent invariant, since the ratio of two such invariants is in $\mathbb{R}(\mathcal{X}_R)^{<\gamma_Q >}$. The corollary now follows from realizing that there are at most $n-r$ algebraically independent monomials in $K$.

\end{proof}

\begin{rem}
The evaluation maps could be interpreted as projective maps of varieties which have the corresponding invariant rings as function fields. However, it is not clear that these varieties actually exist. 
\end{rem}

\section{Invariants Of Ensembles Associated With Surfaces}\label{sec:geometry}

In many cases, cluster ensemble invariants have a geometric interpretation coming from the relationship between cluster ensembles and the Teichm\"uller theory of surfaces. This interpretation is vital for the analysis of invariants on affine ensembles done in section \ref{sec:affine}.  We will assume knowledge about hyperbolic structures on surfaces including monodromy operators, developing maps, and monodromy representations. 

The functions constructed in this section will be invariant under the action of a subgroup of the mapping class group of the surface, denoted $\MCG(S)$. We will denote the subgroup of puncture preserving mapping classes $\PMCG(S)$.

\subsection{Combinatorial data}

Let $S$ be a surface of genus $g$, with $p$ punctures, $b$ boundary components, and $m$ marked points on the boundary. We require that each boundary has at least 1 marked point and that $-\chi(S)= 2g-2 + p + b >0 $. By a hyperbolic structure on $S$, we mean a complete hyperbolic metric on $S$ with geodesic boundary and cusps at each of the marked points and punctures. The Teichm\"uller space of $S$, denoted $T(S)$, parameterizes these structures. 

Given $\Delta$ a triangulation of $S$, we associate a quiver, $Q_\Delta$, as follows: For each edge $e \in \Delta$ we add a node $N_e$ and for each triangle $t \in \Delta$ we add a clockwise oriented cycle of arrows between the nodes associated with the edges of $t$. The nodes associated to boundary edges are frozen. There are  $-3\chi(S)+2m $ total nodes and $m$ frozen nodes.  We refer to \cite{Tri:1} for details about the combinatorics of cluster ensembles associated to triangulations of surfaces.

The cluster ensemble generated by using $Q_\Delta$ as a seed is closely related to the Teichm\"uller theory of $S$.
The coordinates of $\mathcal A_{Q_\Delta}$ and $\mathcal X_{Q_\Delta}$ spaces correspond to lambda lengths and cross ratio coordinates respectively. We refer to \cite{Penner:1} for details about lambda lengths and cross ratios in the context of cluster algebras, and to \cite{Fockgonch:3} section 4 for a simple introduction to the ingredients of a cluster ensemble associated to a hyperbolic surface. We will briefly discuss these ideas here. 

\subsection{$\mathcal{A}$ coordinates and lambda lengths}
$\mathcal{A}_{|Q_\Delta|}$ is Penner's decorated Teichm\"uller space, denoted $\tilde{T}(S)$, see \cite{Penner:decorated}. This space parameterizes hyperbolic structures on $S$ along with a ``decoration'' of $S$ consisting of a choice of horocycles around each puncture and marked point. The $\mathcal{A}$ coordinates are related to this decorated hyperbolic structure by 
\begin{equation}
    a_e = \exp(l_e/2) = \lambda_e
\end{equation}
where $l_e$ is the length of a geodesic representative of $e$ measured between the horocycles. $\lambda_e$ is called the ``lambda length'' of $e$. 

\subsection{$\mathcal{X}$ coordinates and cross ratios}
$\mathcal{X}_{|Q_\Delta|}$ is the ``Teichm\"uller $\mathcal{X}$ space'' of section 4.1 in \cite{Fockgonch:3}. The $\mathcal{X}$ coordinates are related to a hyperbolic structure on $S$ by cross ratios. We have an $\mathcal{X}$ coordinate for each edge $e \in \Delta$ on a non boundary component of $S$, and we associate them with cross ratios as follows. Given a hyperbolic structure on $S$, we may consider the image of the developing map $\Dev(S) \subset \mathbb{H}^2$ and see that $e$ is contained in a square in $\Dev(S)$. This square is determined by the points $z_1,z_2,z_3,z_4 \in \partial\mathbb{H}^2$, with $z_1$ on $e$ and the $z_i$'s are in clockwise order\footnote{There are really two choices of $z_1$, but they lead to the same cross ratio.}. Identifying $\partial\mathbb{H}^2$ with $\mathbb{RP}^1$, we have that 
\begin{equation}
    x_e = \frac{(z_1-z_2)(z_3-z_4)}{(z_2-z_3)(z_1-z_4)}
\end{equation}

We may obtain $T(S)$ as the sub-manifold of $\mathcal{X}_{Q_\Delta}$ where 
\begin{equation}
    \prod_{e \in \Delta, p\in \partial e}x_e = 1
\end{equation}
for every puncture, $p$.

\subsection{Action of the Mapping Class Group}

We can define an action of $\MCG(S)$ on the triangulations of $S$ and hence identify the mapping class group as a subgroup of the cluster modular group. We give an explicit construction of this subgroup here.

Given $ f \in \MCG(S) $ we can define $\gamma_f \in \Gamma $ as follows: $f$ gives a new triangulation of $S$ and hence by Whitehead's lemma there is a path of flips, $P_f$, taking $\Delta$ to $f(\Delta)$. $f$ defines a map between the edges of $\Delta$ and $f(\Delta)$ and it preserves the adjacency relations between the triangles of $\Delta$. This means that $\Delta$ and $f(\Delta)$ have the same associated quivers. $P$ defines a map between the nodes of $Q_{f(\Delta)}$ and $P(Q_\Delta)$ since these quivers come from the same triangulation.  Let $\sigma_{f,P}$ be the isomorphism of quivers $Q_\Delta$ to $P(Q_\Delta)$ defined by the composition 
\begin{equation}
   \sigma_{f,P} : Q_\Delta \xrightarrow{f} Q_{f(\Delta)} \xrightarrow{P} P(Q_\Delta) 
\end{equation}
Thus to $f$ we associate $\gamma_f = \{P_f,\sigma_{f,P}\}$. 

It is not immediately clear that this does not depend on the choice of path, $P$. Let $\{P,\sigma\}$ and $\{R,\tau\}$ be two possible representatives of $\gamma_f$. Then we have 
\begin{equation}
\{P,\sigma\}\{R,\tau\}^{-1} = \{P,\sigma\}\{\tau^{-1}(R^{-1}), \tau^{-1}\} = \{P\sigma\tau^{-1}(R^{-1}),\sigma\tau^{-1}\}
\end{equation}
We need to show that this element is a trivial cluster transformation. $\sigma\tau^{-1}$ is the quiver isomorphism from $R(Q_\Delta)$ to $P(Q_\Delta)$ coming from the fact that these both correspond to the same triangulation of $S$. The composite mutation path, $P\sigma\tau^{-1}(R^{-1})$, consists of following $P$ and then following $R^{-1}$ back to our initial cluster. This introduces a permutation on the cluster variables determined by the map $\tau\sigma^{-1}: P(Q_\Delta) \rightarrow R(Q_\Delta)$. Together these permutations act trivially on the cluster variables, and $\gamma_f$ is well defined in the cluster modular group.

For all but finitely many quivers associated with surfaces, the cluster modular group is essentially equal to the mapping class group, see \cite{Bridgeland:mappingclassgroup} proposition 8.5. For the remaining surfaces, one may check case by case that $\MCG(S)$ is always a normal subgroup of $\Gamma$. 

\subsection{Producing $\mathcal A$ Invariants}
We can produce $\mathcal A$ invariants by clever use of the following fact about lambda lengths. Let $S$ be a hyperbolic surface decorated with horocycles,  $t \in \Delta$ be a triangle, and let $h_3$ be the segment of the horocycle bounded between edges 1 and 2, see figure \ref{fig:lambdalength}. Then the hyperbolic length $h_3$ is $\dfrac{\lambda_3}{\lambda_1\lambda_2}$, see lemma 4.4 of \cite{Penner:1}. 

We can use this formula to write a function that evaluates to the length of a horocycle about a marked point on our surface in terms of the $\mathcal A$ coordinates of a given triangulation. This function is preserved under the action of mapping classes since the total length of the horocycle and the topology of the triangulation are preserved. It is important to notice that this expression will not be invariant under any mapping class that permutes this marked point with another. 

We can construct an invariant for the entire mapping class group by summing all of the lengths of horocycle segments on $S$.  Let $e_{t1},e_{t2},e_{t3}$ be the three edges of a triangle $t \in \Delta$. This invariant is given by the formula
\begin{equation}
    \sum_{t \in \Delta} \frac{a_{e_{t1}}^2+a_{e_{t2}}^2+a_{e_{t3}}^2}{a_{e_{t1}}a_{e_{t2}}a_{e_{t3}}}
\end{equation}

\begin{figure}[hb]
    \centering
    \includegraphics[scale=.5]{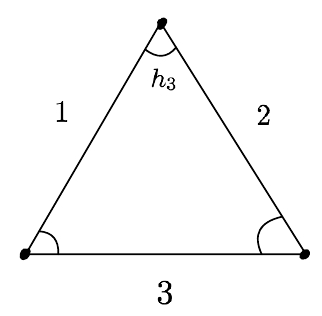}
    \caption{An ideal hyperbolic triangle with horocycle segments around its three vertices.}
    \label{fig:lambdalength}
\end{figure}

\subsection{Producing $\mathcal X$ Invariants}
$\mathcal X $ invariants can be found by taking traces of monodromy matrices associated to paths on the surface that are invariant under the desired subgroup. The monodromy around curves on the surface can be explicitly described in terms of the $\mathcal X$ coordinates of a fixed triangulation. If an arc is preserved by a mapping class, then the edges that the arc intersects and the order in which the arc intersects them must be preserved too. Thus, if we take the trace of the monodromy operator associated to the monodromy around this arc as a function of the $\mathcal X$ coordinates, we must get an invariant of the $\mathcal X$ space under the action of that particular mapping class. For explicit formulas of the traces of monodromy operators in terms of the $\mathcal X$ coordinates see the proof of theorem 6.2 in \cite{Penner:1}.
Functions of $\mathcal X$ coordinates produced in this way will only be invariant under mapping classes that preserve the loop about which this monodromy operator is associated with. Taking loops homotopic to boundary components will always give an invariant for the entire mapping class group. 

We can easily produce Casimir invariants associated to each puncture on $S$ by taking the product of the $\mathcal X$ coordinates associated to edges touching the puncture. This is essentially equivalent to computing the trace of a monodromy operator associated with a loop about a puncture. 

\begin{rem} \label{rem:glue}

We note that in each of these cases that the functions produced are essentially the gluing equations for a hyperbolic structure. Thus one can interpret the invariance of these expressions as a way of stating that the action of the mapping class group does not change the topological structure of the surface and hence does not change the conditions that must be satisfied for there to be a hyperbolic structure on it. 

\end{rem}

\subsection{Examples}
For the first two examples, we refer to \cite{Margalit:PMCG} section 2 for computations involving the mapping class group of selected surfaces. 

\begin{exmp}
\label{exmp:hypmarkov}
There is a simple geometric interpretation of the invariant of the Markov quiver, example \ref{exmp:markov}. If $S = S_{1,1}$ is a torus with one puncture then the cluster ensemble associated to $S$ is of type $ A_1^{(1,1)} $ see figure \ref{fig:torus_quiver}. If we have a hyperbolic metric on $S$ with the puncture at infinity and a horocycle around the puncture, then the theory of lambda lengths implies that the function $F$ is simply the formula for half the length of the horocycle in terms of the $\mathcal A$ coordinates. Since the length of the horocycle is independent on the triangulation and there is only one topological type of triangulation, $F$ must be invariant under exchanges. 

\begin{figure}
    \centering
    \includegraphics[scale =.5]{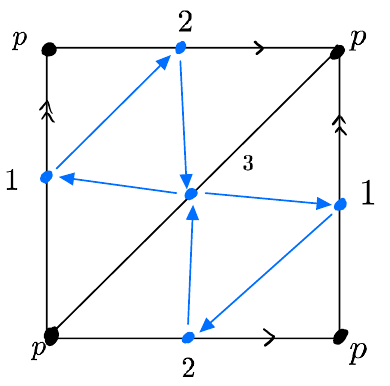}
    \caption{The surface $S_{1,1}$ with a choice of triangulation and associated quiver.}
    \label{fig:torus_quiver}
\end{figure}

\end{exmp}

\begin{exmp}
\label{exmp:4punctured}
The invariants of example \ref{exmp:d411:1} can be understood via surfaces as well. The cluster algebra associated to $S=S_{0,4}$, a  four punctured sphere, is of type $D_4^{(1,1)}$. If we take a triangulation of $S$ that topologically looks like a tetrahedron, then the quiver associated with this triangulation is $Q_4$ from figure \ref{fig:D4-11_quivers} and the length of a particular horocycle around a puncture is given by a function like $F_{456}$. We have that 
\begin{equation}
\PMCG(S) = \mathbb{Z}*\mathbb{Z}
\end{equation}
and hence we have that this function is invariant under the subgroup claimed in the example. 

%The relationship with the previous example can be understood by the relationship between $S_{0,4}$ and $S_{1,1}$ via the hyper elliptic involution:
%\begin{equation}
%    \iota: S_{1,1} \rightarrow S_{1,1}
%\end{equation}
%If we identify $S_{0,4} $ with $ S_{1,1}/\iota(S_{1,1}) $
\end{exmp}

\begin{exmp} \label{exmp:hypa11}

The $\mathcal X$ invariant from example \ref{exmp:a11} can be interpreted as follows: The $\widetilde{A}_1$ affine ensemble is associated to an annulus with one marked point on each boundary component (see figure \ref{fig:surf1}). The cluster modular group corresponds exactly to the mapping class group of $S$ and the generator $\gamma= \{1,(12)\}$ corresponds to a Dehn twist about $\delta$. If we take the trace of the monodromy operator, $\rho(\delta)$, associated to this surface we get $ Tr(\rho(\delta)) = \dfrac{x_2(x_1+1)+1}{\sqrt{x_1x_2}}$. This is exactly the square root of $G(x_1,x_2)$ from before. 

\begin{figure}[ht]
    \centering
    \includegraphics[scale=.5]{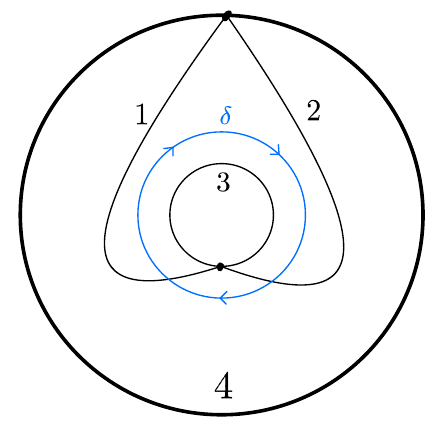}
    \caption{The surface S, along with a choice of triangulation. $\delta$ is the generator of the fundamental group and the mapping class $\gamma$ corresponds to a Dehn twist about $\delta$.}
    \label{fig:surf1}
\end{figure}

We can also investigate the limiting behavior of the cluster variables via the hyperbolic geometry of $S$. The $\mathcal{A}$ coordinates correspond to lambda lengths associated to the two interior arcs on $S$. After many applications of $\gamma$, we saw that the $\mathcal{A}$ were approximately changing by multiples of $ \lambda = \frac{F+\sqrt{F^2-4}}{2} = \exp(\arccosh(F/2))$, where $F= \rho^*(\sqrt{G})$. One may check that there is exactly 1 closed geodesic with the same homotopy class as $\delta$ and its hyperbolic length is equal to $2\arccosh(F/2)$. Thus, the multiplier $\lambda$ is essentially the lambda length of the arc $\delta$.

\end{exmp}

\section{Invariants on Affine Ensembles}\label{sec:affine}

We will give a classification of the invariants for the puncture preserving mapping class group on each of the affine type cluster ensembles with trivial coefficients. These are the ensembles associated to mutation classes of quivers which contain orientations of affine $ADE$ Dynkin diagrams. In the affine $\widetilde{A}_n$ case there is really a family of quiver mutation classes $\widetilde{A}_{p,q}$ with $p+q=n+1$ and $0 < q \leq p$. These mutation classes have a quiver which is an $n+1$ cycle with $p$ arrows one direction and $q$ the other.   

The $\widetilde{A}_{p,q}$ and $\widetilde{D_n}$ affine type cluster ensembles are associated to an annulus with $p$ and $q$ marked points on each boundary and a twice punctured disk with $n-2$ marked points on the boundary respectively, see \cite{Tri:1} examples 6.9 and 6.10. The $E$ type affine cluster ensembles are not associated to surfaces, but we will introduce an analog of the mapping class group in these cases.

Each of the affine $ADE$ mutation classes contains the quiver $T_{p,q,r}$,  where $\frac{1}{p}+\frac{1}{q}+\frac{1}{r} < 1$, see figure \ref{fig:QuiverGeneralAffine}. For $\widetilde{A}_{p,q}$, we have  $(p,q,r)=(p,q,1)$ , for $\widetilde{D_n}$ we have $(p,q,r)=(n-2,2,2)$, and for $\widetilde{E}_n$ we have $(p,q,r)=(n-3,3,2)$. We will use this quiver as our initial seed. We associate variables
\begin{align}
    &(a_1,a_2,b_2,\dots,b_p,c_2,\dots,c_q,d_2,\dots,d_r) \\
    &(x_1,x_2,y_2,\dots,y_p,z_2,\dots,z_q,w_2,\dots,w_r)
\end{align}
for the variables on the $\mathcal A$ and $\mathcal X$ spaces in a natural way.  

\begin{figure}[ht]
\centering
\begin{tikzcd}[column sep=1.5em]
              &                   &               &                & A_1 \arrow[dd, Rightarrow]           & B_2 \arrow[l]   & ... \arrow[l] & B_{p-1} \arrow[l] & B_{p} \arrow[l] \\
D_r \arrow[r] & D_{r-1} \arrow[r] & ... \arrow[r] & D_2 \arrow[ru] &                                    &                 &               &                   &                 \\
              &                   &               &                & A_2 \arrow[lu] \arrow[ruu] \arrow[r] & C_2 \arrow[luu] & ... \arrow[l] & C_{q-1} \arrow[l] & C_q \arrow[l]  
\end{tikzcd}

    \centering
    
    \caption{General form of a $T_{p,q,r}$ quiver.}
    \label{fig:QuiverGeneralAffine}
\end{figure}
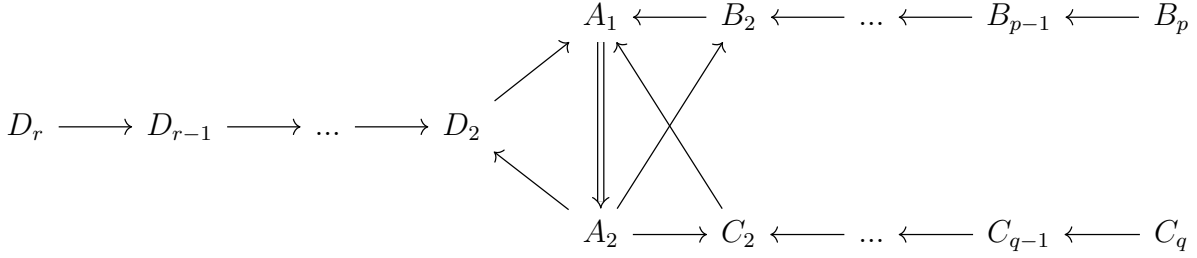

In the $A$ and $D$ cases, the puncture preserving mapping class groups of the corresponding surfaces are isomorphic to $\mathbb{Z}$. We can identify this as a subgroup of the cluster modular groups of $\Gamma_{T_{p,q,r}}$ generated by the element $\gamma=\{A_1,(A_1A_2)\}$. This also gives a subgroup isomorphic to $\mathbb{Z}$ in the $E$ cases and we will take this as our analog to the mapping class group. For the rest of this section, we will always be considering invariants with respect to the group $\Gamma = <\gamma>$.
 
We can now state the main theorem of this section.

\begin{thm}
On an affine cluster ensemble with trivial coefficients of rank $n+1$, there are exactly $n$ algebraically independent invariant functions under the action of the puncture preserving mapping class group on the $\mathcal X$ and $\mathcal A$ spaces. Explicitly, $\mathbb{R}(\mathcal X_{T_{p,q,r}})^\Gamma$ is generated by the functions $G(x_1,x_2)$ of example \ref{exmp:a11}, $y_2(x_2(x_1+1)+1)$ , $z_2(x_2(x_1+1)+1)$ , $w_2(x_2(x_1+1)+1)$, and the remaining $\mathcal X$ coordinates corresponding to nodes that are not connected to nodes $A_1$ or $A_2$. $\mathbb{R}(\mathcal A_{T_{p,q,r}})^\Gamma$ is generated by $\rho^*(\sqrt{G(x_1,x_2)})$ and the remaining $n-1$ $\mathcal A$ coordinates associated to the nodes other than nodes $A_1$ and $A_2$. 
\end{thm}

We divide this theorem into 4 lemmas. 

\begin{lem}  \label{lem:xa11}
$\mathbb{R}(\mathcal X_{T_{1,1,1}})^\Gamma= \mathbb{R}(G(x_1,x_2)) $ where $ G(x_1,x_2) = \dfrac{(x_2(x_1+1)+1)^2}{x_1x_2}$ 
\end{lem}
\begin{proof}

Let $Q=T_{1,1,1}$. This is the same quiver from examples \ref{exmp:a11} and \ref{exmp:hypa11}.
This lemma follows from further geometric analysis of the function $G$.
For simplicity of notation, let $x=x_1$ and $y=x_2$

Let $g(x,y) \in \mathbb{R}(\mathcal X_Q)^\Gamma$. Let $S$ be an annulus with one marked point on each boundary component, see \ref{fig:surf1}, and $Q = Q_\Delta = T_{1,1,1}$ for any triangulation of $S$. Then, since there are no punctures on $S$, $\mathcal{X}_{|Q|}$ is isomorphic to $T(S)$. Let $\Sigma = \Hom(\mathbb{Z}, \PSL(2,\mathbb{R}))//\PSL(2,\mathbb{R})$ be the character variety associated to the holonomy of a hyperbolic structure on $S$, and let $T$ be the trace coordinate on $\Sigma$. We refer to \cite{Goldman:fricke} for an introduction to trace coordinates on representation varieties of this type. There is a map $\pi$ From $T(S)$ to $\Sigma$ taking a hyperbolic surface to its monodromy representation, $\rho$. By the analysis of section 5, we have that $G(x,y)= \pi^*(T^2)$. The lemma follows by exhibiting a rational lift, $\hat{g}: \Sigma \rightarrow \mathbb{R}$ such that $g = \hat{g} \circ \pi$. This is encompassed by the following commutative diagram; the map $s_c$ is defined below. 
\begin{equation}
    \centering
    \begin{tikzcd}[sep = large]
T(S) \arrow[d, "\pi"', bend right] \arrow[r, "g"]           & \mathbb{R} \\
\Sigma \arrow[u, "s_c"', bend right] \arrow[ru, "\hat{g}"'] &           
\end{tikzcd}
\end{equation}

\begin{figure}[ht]
    \centering
    \includegraphics[scale=.2]{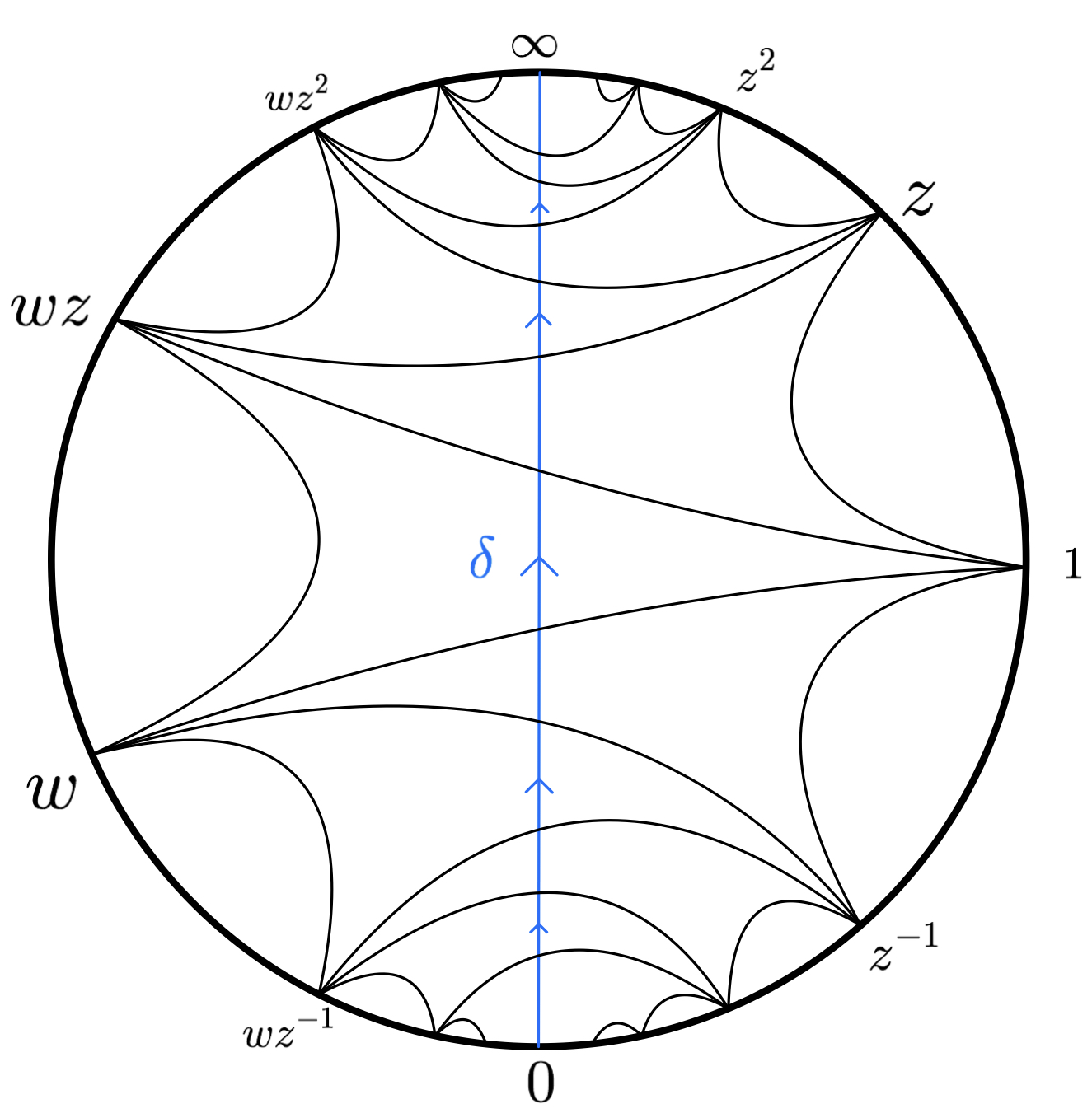}
    \caption{The image of $\Dev(S)$ in $\mathbb{H}^2$. We are identifying $\partial \mathbb{H}^2$ with $\mathbb{RP}^1$ in the usual way. The values of $x$ and $y$ are given by the cross ratios of the squares containing their corresponding edges.}
    \label{fig:hyp1}
\end{figure}

To construct $\hat{g}$, we construct a rational splitting $s: \Sigma \rightarrow T(S)$ satisfying $\pi \circ s = id_\Sigma$ and define $\hat{g}= g \circ s$. The lemma follows by showing that $g= g \circ s \circ \pi$ and that $s^*(\mathbb{R}(x,y)) \subset \mathbb{R}(T^2)$. 

The fibers of $\pi$ are 1 dimensional and isomorphic to $\mathbb{R}^{>0}$. By using hyperbolic isometries, we may arrange that the image of the developing map, $\Dev(S) \subset \mathbb{H}^2$ appears as figure \ref{fig:hyp1} . $\delta$ generates $\pi_1(S)$ and $\rho(\delta)= 
\begin{bmatrix}
 \sqrt{z} & 0\\
 0 & \frac{1}{\sqrt{z}}
\end{bmatrix}
$.
The fibers of $\pi$ are parameterized by the value of $w$. Thus for any $c \in \mathbb{R}^{>0}$ we have a possibly non-rational splitting $s_c$ given by picking the surface with $w=-c$. 

We have the following fact: $g(x,y)$ does not depend on the value of $w$. To see this, we use the invariance of $g$. 
Let $(x,y)=(x_0,y_0)$ be our initial cluster and let $(x_k,y_k)$ be the cluster obtained after applying $\gamma^k$. We can compute the limiting behavior of $\{x_k\}$ and $\{y_k\}$ by observing what happens to triangulation of $S$ after applying $\gamma^k$. We can compute that
\begin{equation}
    x_k \rightarrow \infty \quad y_k \rightarrow 0 \quad x_ky_k \rightarrow z.
\end{equation}
This follows either by looking at figure \ref{fig:hyp2} or by using the analysis of the limiting sequences of $\mathcal{A}$ coordinates in example \ref{exmp:a11} combined with the map $\rho: \mathcal{A}_Q \rightarrow \mathcal{X}_Q$.
Thus, writing $g(x,y)=g'(\alpha,\beta)$ where $ \alpha = x^{-1}$ and $\beta= xy$ we can see, using the invariance of $g$, that $g(x,y)= g'(0,z)$ and only depends on $z$ and is independent of $w$. This implies that $g \circ s_c \circ \pi = g$ for any choice of $c$. 

The only thing left to prove is that $g$ is actually a rational function in $T^2$. To see this, pick $s= s_1$ as our splitting. Let $\tilde{x},\tilde{y}$ be the $\mathcal{X}$ coordinates of the surface with $w=-1$. We have 
\begin{equation}
    \tilde{x} = \dfrac{4z}{(1-z)^2} \quad \tilde{y}= (\dfrac{1-z}{1+z})^2 \quad \pi^*(T) = \dfrac{z+1}{\sqrt{z}}
\end{equation}
and finally 
\begin{equation}
    \tilde{x} = ((\frac{T}{2})^2-1)^{-1} \quad \tilde{y}= 1- (\frac{2}{T})^2.
\end{equation}
So $\tilde{x}$ and $\tilde{y}$ are rational functions of $T^2$ and so $g(x,y) = (g \circ s_1 \circ \pi )(x,y) = g(\tilde{x},\tilde{y})$ is. 
\end{proof}

\begin{figure}[ht]
    \centering
    \includegraphics[scale=.2]{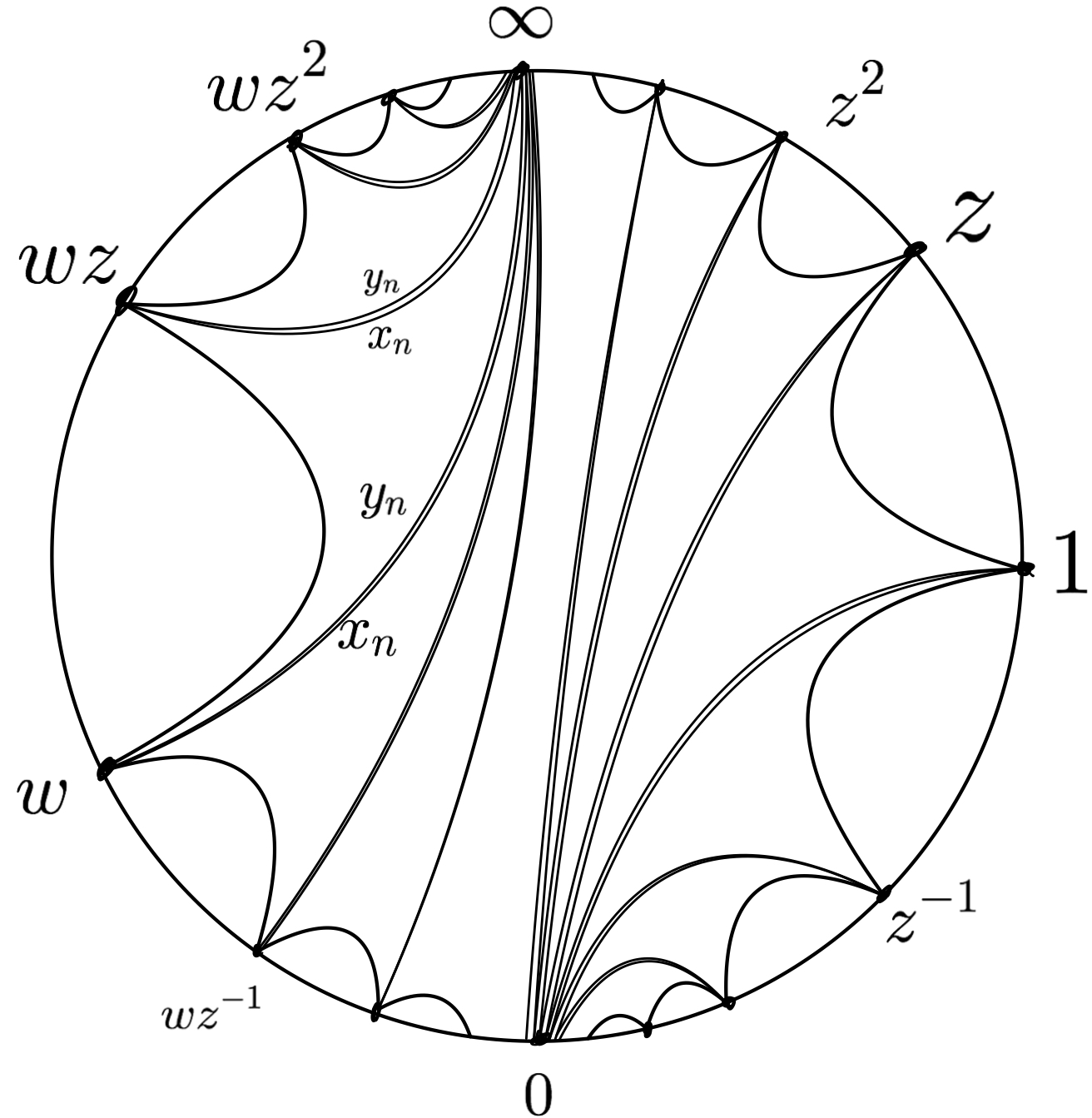}
    \caption{The triangulation of S obtained after applying $\gamma^k$ for $k>>0$. $x_k$ and $y_k$ correspond to nearly the same arc on $S$ and $x_ky_k$ is approximately the cross ratio of the square with endpoints $(0,z^2,z,1)$.}
    \label{fig:hyp2}
\end{figure}

\begin{cor} \label{cor:x}
$\mathbb{R}(\mathcal X_{T_{1,1,1}}, \sqrt{x_1x_2})^\Gamma= \mathbb{R}(\sqrt{G(x_1,x_2)}) $
\end{cor}
This follows by the same argument, noting that $\sqrt{x_1x_2}$ is rational function of $T$. 

\begin{lem} \label{lem:a221}
$\mathbb{R}(\mathcal A_{T_{2,2,1}})^\Gamma$ is generated by 
\begin{equation}
    \{\rho^*(\sqrt{G(x_1,x_2)}) = \dfrac{a_1^2+a_2^2+b_2c_2}{a_1a_2},b_2,c_2\}
\end{equation}
\end{lem}
\begin{proof}

We apply theorem \ref{thm:adims} with $Q=T_{2,2,1}$ and $R \subset Q$ having $N(R^\mu) = \{A_1,A_2\}$. Then we can see that $\rho^*(\sqrt{G(x_1,x_2)}) $ along with the coefficients of $\mathcal{A}_R$ gives the basis of the lemma. We need to show that any rational invariant is a rational function in this basis. 

Let $f \in \mathbb{R}(\mathcal{A}_R)^\Gamma$ be an invariant function. We can see that $\rho^*$ maps $x_1$ and $x_2$ to $\dfrac{a_2^2}{b_2c_2}$ and $\dfrac{b_2c_2}{a_1^2}$ respectively. Thus 
\begin{equation}
    \rho^* : \mathbb{R}(\mathcal{X}_R , \sqrt{x_1x_2)})^\Gamma(b_2,c_2) \rightarrow \mathbb{R}(\mathcal{A}_R)^\Gamma
\end{equation}

The lemma follows by showing that this map is an isomorphism and applying corollary \ref{cor:x}.

Let $\rho^*(\mathbb{R}(\mathcal{X}_R , \sqrt{x_1x_2)})(b_2,c_2))=K$. We have
\begin{equation}
    \mathbb{R}(a_1^2,a_2^2,a_1a_2,b_2,c_2) = K \subset \mathbb{R}(\mathcal{A}_R) 
\end{equation}
is a degree 2 field extension. We want to show that $ f \in K$. Let $f= a_1g+h$ where $g$ and $h$ are in $K$.

Following a similar argument to the proof of lemma \ref{lem:xa11}, we associate $\mathcal A_{R}$ with the decorated Teichm\"uller space $\tilde{T}(S)$ where $S$ is an annulus with two marked points. Then by the same argument as before, the value of $f(a_1,a_2,b_2,c_2)$ does not depend on the value of $w$. So $f(a_1,a_2,b_2,c_2) = f(\tilde{a_1},\tilde{a_2},b_2,c_2)$ where $\tilde{a_1},\tilde{a_2}$ are the lambda lengths associated to the surface with the same choice of horocycles and $w=-1$

We have 
\begin{equation}
    \tilde{a_1}=\sqrt{\frac{\rho^*(\tilde{x_2})}{b_2c_2}}=\sqrt{(b_2c_2(1-(\frac{2}{\pi^*(T)})^2))^{-1}}=\sqrt{(b_2c_2(1-(\frac{2a_1a_2}{a_1^2+a_2^2+b_2c_2})^2))^{-1}}
\end{equation}

which implies that we may write $\tilde{a_1}$ in terms of $a_1$ but with square roots in the expression. This shows that $g=0$ since $f$ is a rational function. Thus $f \in K$ and the lemma follows

\end{proof}

\begin{lem}
$\mathbb{R}(\mathcal X_{T_{p,q,r}})^\Gamma$ is generated by the functions $G(x_1,x_2)$ of example \ref{exmp:a11}, $y_2(x_2(x_1+1)+1)$ , $z_2(x_2(x_1+1)+1)$ , $w_2(x_2(x_1+1)+1)$, and the remaining $\mathcal X$ coordinates corresponding to nodes that are not connected to nodes $A_1$ or $A_2$.
\end{lem}
\begin{proof}
This follows from lemma \ref{lem:xa11}, theorem \ref{thm:xdims} and corollary \ref{cor:xdims}. We observe that the set of functions $\{G(x_1,x_2), y_2(x_2(x_1+1)+1), z_2(x_2(x_1+1)+1), w_2(x_2(x_1+1)+1)\}$ along with the remaining $\mathcal X$ coordinates are each elements of $\mathbb{R}(\mathcal X_{T_{p,q,r}})^\Gamma$ and by corollary \ref{cor:xdims} there cannot be any more independent generators. By theorem \ref{thm:xdims} applied with $R = T_{1,1,1} \subset Q = T_{p,q,r}$ and lemma \ref{lem:xa11}, any invariant that only depends on $x_1,x_2$ must be generated by $G(x_1,x_2)$. Finally, we note that each of the elements of the given basis is of degree 1 in the variables other than $\{x_1,x_2\}$ and therefore they must generate $\mathbb{R}(\mathcal{X}_Q)^{<\gamma_Q>} / \mathbb{R}(\mathcal{X}_R)^{<\gamma_R>}$.
\end{proof}

\begin{lem}
$\mathbb{R}(\mathcal A_{T_{p,q,r}})^\Gamma$ is generated by $\rho^*(\sqrt{G(x_1,x_2)})$ and the remaining $n-1$ $\mathcal A$ coordinates associated to the nodes other than nodes $A_1$ and $A_2$. 
\end{lem}
\begin{proof}
There are several cases to check. $S$ be an annulus with one marked point on each boundary, as before and $\tilde{T}(S)$ be the decorated Teichm\"uller space of $S$. 

When $(p,q,r) = (1,1,1)$ or $(2,1,1)$, we have that 
\begin{equation}
    \mathcal{A}_{T_{1,1,1}} \subset \mathcal{A}_{T_{2,1,1}} \subset \mathcal{A}_{T_{2,2,1}} = \tilde{T}(S)
\end{equation}
as the subsets parameterizing surfaces with both boundary lambda lengths or one boundary lambda length equal to 1 respectively. The lemma follows for these cases by a similar argument to lemma \ref{lem:a221}.

When $(p,q,r)=(2,2,2)$ we notice that if we let $\tilde{c_2}=c_2d_2$, then the action $\Gamma$ on the $\mathcal{A}$ coordinate functions is exactly the same as its action on the $\mathcal{A}_{T_{2,2,1}}$ space. Thus, $\mathbb{R}(\mathcal A_{T_{2,2,2}})^\Gamma = \mathbb{R}(\mathcal A_{T_{2,2,1}})^\Gamma(d_2) $ by the map $ \tilde{c_2} \rightarrow c_2$

Finally, the lemma follows for any $(p,q,r)$ by noticing that the coordinates with indices larger than 2 do not affect the action of $\Gamma$. Let $\alpha$ be the set of these coordinates and let $V = T_{\min(p,2),\min(q,2),\min(r,2)}$. Then it is clear that
\begin{equation}
    \mathbb{R}(\mathcal A_{T_{p,q,r}})^\Gamma = \mathbb{R}(\mathcal A_V)^\Gamma(\alpha)
    \qedhere
\end{equation}
\end{proof}

\begin{rem}
Nowhere in the proof of this theorem did we use the fact that $p^{-1}+q^{-1}+r^{-1} > 1$. Thus the theorem holds for any type $T_{p,q,r}$ cluster ensemble. When $p^{-1}+q^{-1}+r^{-1} = 1$, $T_{p,q,r}$ is a quiver associated to a doubly extended root system of $E$ type.
\end{rem}

%\begin{rem}
%The cluster modular groups of the affine type cluster ensembles are all isomorphic to $\mathbb{Z}\times G$ where $G$ is some finite group. The element $\gamma $ generates a subgroup $d\mathbb{Z} \subset \mathbb{Z}$ where $d = lcm(p,q), 2n, 6, 12, 30 $ for types $\tilde{A_{p,q}}, \tilde{D_n}, \tilde{E_6}, \tilde{E_7}, \tilde{E_8}$ respectively. In this way we can see that 
%\end{rem}

\section{Conjectures about the general structure of invariants}\label{sec:conjectures}

There are many questions we can ask about invariant functions for general cluster ensembles. For example, are there non trivial rational invariants for any subgroup of $\Gamma$?. This is almost certainly negative, but we may at least hope that there is some condition that ensures the existence of invariants. 

\subsection{Further Geometric structures related to invariants}
We suspect that the analysis of section 5 extends to cluster ensembles associated to the higher Teichm\"uller spaces of Fock and Goncharov in \cite{FockGonch:1}. The mapping class group of the surface will again be a canonical subgroup of the cluster modular group of these ensembles. As before, it is possible to compute traces of monodromy operators about loops on the surface in terms of the $\mathcal{X}$ coordinates, and these will give $\mathcal{X}$ invariants for mapping classes that preserve the given loop. 

There is also evidence that invariant functions may aid in the analysis of the gluing equations for more general geometric structures on manifolds, generalizing remark \ref{rem:glue}. The work of Nagao, Terashima, and Yamazaki in \cite{hyperbolic3man} gives indications of this via their notion of a ``parameter periodicity equation''. 

\subsection{Correspondence between $\mathcal A$ and $\mathcal X$ Invariants}

We will give some evidence for a correspondence between the invariants of the $\mathcal A$ and $\mathcal X$ spaces in terms of denominator vectors.

\begin{defn}
The \emph{denominator vector} of a Laurent polynomial is the vector of exponents appearing in the denominator of the polynomial written as a single fraction. 
\end{defn}

\begin{defn}

We call a pair of bases or partial bases for $\mathcal A$ and $\mathcal X$ invariants that correspond via denominator vectors a \emph{correspondence basis}.

\end{defn}

Through all of the examples of section 3, we can explicitly see examples of correspondences between the $\mathcal A$ and $\mathcal X$ invariants. In the cases where no $\mathcal X$ invariants are given, there are Casimir $\mathcal X$ invariants corresponding via denominator vectors to each of the $\mathcal A$ invariants. However, it is not clear in most of the examples that the functions given are actually a full basis of invariants.
\begin{exmp}
Explicitly, this correspondence for the invariants of example \ref{exmp:d411:1} are as follows:
\begin{align*}
    \dfrac{(a_1a_4+a_2a_5+a_3a_6)^2}{a_1a_2a_3a_4a_5a_6} & \xleftrightarrow{\hspace{1cm}} (x_1x_2x_3x_4x_5x_6)^{-1} \\
    \dfrac{(a_1a_4+a_2a_5+a_3a_6)}{a_4a_5a_6} & \xleftrightarrow{\hspace{1cm}} (x_4x_5x_6)^{-1} \\
    \dfrac{a_1}{a_4} & \xleftrightarrow{\hspace{1cm}} \dfrac{x_1}{x_4}
\end{align*}

\end{exmp}

It is possible to give ample evidence for this correspondence in the case that our algebra is associated to a surface. It is easy to see that for any puncture on our surface, the formula for the length of the horocycle about that puncture and the Casimir $\mathcal X$ invariant for the puncture correspond via denominators. Similarly, we get a correspondence of denominators between squared traces of monodromies about boundary components and horocycle lengths about marked points on the boundary. 
Furthermore, theorem 6.1 shows that there is a correspondence basis for the invariants for the mapping class group on all of the affine cluster ensembles.

\subsection{Laurent Property of Invariants}

Many of the examples of $\mathcal{A}$ invariants are Laurent polynomials in the $\mathcal{A}$ coordinates. This is not true in every case, as the functions of example \ref{exmp:d411:1} include inverses of Laurent polynomials too. However we may conjecture the following: 

\begin{con}
    There is a basis of $\mathbb{R}(\mathcal{X}_Q)^{\Gamma_\circ}$ and corresponding basis of $\mathbb{R}(\mathcal{A}_Q)^{\Gamma_\circ}$ such that $\Gamma$ acts on the basis of  $\mathbb{R}(\mathcal{A}_Q)^{\Gamma_\circ}$ by positive Laurent polynomials. 
\end{con}

The examples of section 3, other than example \ref{d411:2}, give ample evidence for this conjecture. This conjecture also implies, following theorem 6.1, that all of the $\mathcal A$ coordinates appearing on the tails of the $T_{p,q,r}$ quiver are Laurent polynomials in the initial invariants. This is easy to verify case by case, and is probably not too difficult to prove in general. 

It seems that in some cases there is a stronger version of this phenomenon. We may occasionally find a correspondence basis such that the action of the cluster modular group sends the basis to Laurent \emph{monomials} in the basis. In other words, we can pick a basis that gives a representation 
\begin{equation}
    \Gamma/\Gamma_\circ \rightarrow \GL(k,\mathbb{Z})
\end{equation}

Surprisingly, we have that the action of $\Gamma $ on the $\mathcal X$ invariants is also by Laurent monomials and the induced representation seems to be identical in this case. 

\begin{exmp}

Continuing example \ref{G233}, we can take $(F_1,F_2)$ as a possibly partial basis of invariants. Then one may check that the paths $\tau=\{34,(12)\}$ and $r = \{414 ,(132)\}$ generate $D_{12}$ and we have that $\tau((F_1,F_2))=(F_1,F_1/F_2)$ and $r((F_1,F_2))=(F_1/F_2,F_1)$. 
Thus we have the representation 
\begin{equation}
   \pi: D_{12} \rightarrow \GL(2,\mathbb{Z})
\end{equation}
given by 
\begin{equation}
    \pi(\tau)= \begin{bmatrix}
 1 & 1 \\
 0 & -1
\end{bmatrix},
\pi(r)= \begin{bmatrix}
  1 & 1 \\
 -1 & 0
\end{bmatrix}
\end{equation}

\end{exmp}

\newpage
\bibliographystyle{ieeetr}
\bibliography{References}

\end{document}